\documentclass[reqno]{amsart}
\usepackage{preamble}

\author{Juliana Curtis}
\address[\textsc{Curtis}]{University of Washington}
\email{jacca@uw.edu}
\author{Tuong Le}
\address[\textsc{Le}]{Princeton University}
\email{tuongle@princeton.edu}
\author{Chayim Lowen}
\address[\textsc{Lowen}]{Princeton University}
\email{chayiml@princeton.edu}

\title{On smooth combinatorial products of simplices}

\begin{document}

\begin{abstract}
    We show that every lattice weak Minkowski summand of a smooth polytope combinatorially isomorphic to a product of simplices has a quadratic triangulation. This is achieved by identifying this class of polytopes with the class of Nakajima polytopes.
    We combine these results to give a new proof that two lattice weak Minkowski summands of the same such smooth polytope are relatively IDP.
    Along the way, we prove a structure theorem for simple polytopes whose 2-faces are triangles and trapezoids.
\end{abstract}

\maketitle

\section{Introduction}

We will prove a number of theorems relating to smooth polytopes whose combinatorial type is that of a product of simplices.
Our first result is a local-to-global principle for faces of a simple polytope.
It requires no lattice structure.
The remaining results are motivated by covering properties for lattice polytopes.
Recall that a lattice polytope is smooth if the set of primitive edge directions at each of its vertices forms part of a basis of the ambient lattice. 
It is generally anticipated that smooth polytopes have desirable covering properties, such as the integer decomposition property (IDP). 
In this vein, we prove that smooth combinatorial products of simplices and their lattice weak Minkowski summands (which \emph{need not} be smooth) admit regular, flag unimodular triangulations. 

We begin with a purely structural theorem for simple polytopes whose two-dimensional faces are triangles and trapezoids. The statement is best understood as the second in a sequence of four, the rest of which can be found in the literature. It is a statement about \emph{arbitrary real} (i.e.\ not necessarily lattice) simple polytopes.
\begin{Thm}\label[Thm]{thm:2faces}
    Let $P$ be a simple polytope whose 2-faces are triangles and quadrilaterals. 
    \begin{enumerate}
        \item\textup{({\cite{wiemeler15}*{Prop.\ 4.5}, \cite{Yu21}*{Appendix}})} \label[Thm]{thm:combprod} Without any further assumptions, $P$ is combinatorially isomorphic to a product of simplices.
        \item \label[Thm]{thm:combprodsimp} If each quadrilateral 2-face of $P$ is a trapezoid, then $P$ is affinely isomorphic to a Cayley tower (see \Cref{df:Cayley_tower}).
        \item \textup{(\cite{Teddy}*{Theorem 6.19})}
        \label[Thm]{thm:paraprod}
        If each quadrilateral 2-face of $P$ is a parallelogram, then $P$ is affinely isomorphic to a product of simplices.
        \item\textup{({\cite{coxeter}*{Lemma 2.7}})}\label[Thm]{thm:orthoprod} If each quadrilateral 2-face of $P$ is a rectangle, then $P$ is isometric to a product of simplices. 
    \end{enumerate}
\end{Thm}
\noindent
Part (b) of this theorem, which is our contribution, is a subtle interpolation between parts (a) and (c). 
As we shall see, Cayley towers are a very controlled way to build polytopes with the combinatorial type of a product of simplices, obtained by recursively taking Cayley sums of smaller Cayley towers.
We now switch gears and discuss \emph{lattice} polytopes.
In a smooth lattice polytope $P$, each quadrilateral 2-face is a trapezoid by \cite{fulton}*{Section 2.5, Exercise (a)}. A minor variation on the proof of \Cref{thm:2faces}\ref{thm:combprodsimp} 
gives:
\begin{Thm}\label[Thm]{thm:struct_thm_smooth}
    Let $P$ be a smooth lattice polytope combinatorially isomorphic to a product of simplices. Then $P$ is lattice isomorphic to an integral Cayley tower (see \Cref{df:Cayley_tower}).
\end{Thm}
\noindent
\Cref{thm:struct_thm_smooth} may instead be deduced from a theorem of Batyrev \cite{batyrev91}*{Corollary 4.4} on smooth complete toric varieties. A closely related result is proven in Dobrinskaya \cite{classification} and Choi--Masuda--Suh \cite{quasitoric2}, who instead work with quasitoric manifolds; see \Cref{sec:literature}. Both proofs use smoothness in an essential way, so do not extend to give a proof of 
\Cref{thm:2faces}\ref{thm:combprodsimp}.

We go on to use \Cref{thm:struct_thm_smooth} to characterize the \emph{lattice} weak Minkowski summands of smooth combinatorial products of simplices as being precisely the \emph{Nakajima} polytopes. These arose in \cite{Nakajima} in the study of affine toric varieties with mild singularities and were later studied in \cite{haase21}*{\textsection~2.2}. Very recently, Nakajima polytopes were used by Ferroni \cite{F26} to disprove a number of related conjectures in Ehrhart theory, most notably Stanley's conjecture on the unimodality of the $h^*$-vectors of IDP polytopes \cites{S86,SV13}. We recall the definition of these polytopes in \Cref{sec:WMS_of_SCPS}. 
\begin{Thm}\label[Thm]{thm:Nakajima}
    For a lattice polytope $P \subseteq \R^n$, the following are equivalent:
    \begin{enumerate}[(1)]
        \item $P$ is a weak Minkowski summand of a smooth combinatorial product of simplices.\label{it:WMS_SCPS}
        \item $P$ is a weak Minkowski summand of a smooth combinatorial cube.\label{it:WMS_SCC}
        \item $P$ is lattice isomorphic to a Nakajima polytope.\label{it:Nakajima}
    \end{enumerate}
\end{Thm}
\noindent
As an immediate corollary of \Cref{thm:Nakajima}, using a result of Haase et al.\ for Nakajima polytopes \cite{haase21}*{Corollary 2.9}, we obtain the following.
\begin{Cor}\label[Cor]{cor:unimodulartriangulation}
    Let $P$ be a smooth combinatorial product of simplices, or a lattice weak Minkowski summand  thereof.
    Then $P$ admits a quadratic triangulation.
\end{Cor}
\noindent
With a little more work, we can apply the standard maneuver known as the \textit{Cayley trick} to prove the following theorem.
Recall that a pair of lattice polytopes $P, Q$ is \emph{relatively IDP} if 
\[
    P \cap \Z^n + Q \cap \Z^n = (P + Q) \cap \Z^n.
\]
A lattice polytope $P$ is IDP if 
the pair $(P, kP)$ is relatively IDP for all natural numbers $k$. 
The following result has appeared in a different guise as \cite{ikeda09}*{Corollary 4.2}. See \Cref{sec:literature} for details of the translation, which relies crucially on \cite{batyrev91}*{Corollary 4.4}.
\begin{Thm}\label[Thm]{thm:absolute_ODA}
    Let $P$ be a smooth lattice polytope combinatorially isomorphic to a product of simplices.
    Let $Q, Q'$ be lattice polytopes which are weak Minkowski summands of $P$.
    Then the pair $(Q, Q')$ is relatively IDP. In particular, every lattice weak Minkowski summand of $P$ is IDP.
\end{Thm}
\noindent
The importance of \Cref{thm:absolute_ODA} lies in its relationship to the many versions of what is known as Oda's Conjecture, posed by Tadao Oda in 1997 and documented in \cite{oda08}. We recast them here in the language of polytopes (and defer further explanation to \cref{sec:literature}).

\begin{Conj}[Oda's Conjecture] \label[Conj]{conj:oda} Let $P$ be a smooth lattice polytope, let $Q$ be a lattice polytope Minkowski equivalent to $P$, and let $R$ be a lattice weak Minkowski summand of $P$.
\begin{enumerate} 
    \item \label{it:oda_weak} $P$ is IDP. 
    \item \label{it:oda_med} $(P, Q)$ is relatively IDP. 
    \item \label{it:oda_strong}$(P, R)$ is relatively IDP. 
\end{enumerate}
\end{Conj}
\noindent
Note that the (universalized) statements (a), (b), (c) are of increasing strength.
To this day, the weakest of these remains open even in dimension 3.
In dimension 2, \cref{conj:oda} \ref{it:oda_strong} was answered affirmatively in \cite{fakhruddin02}. It was further shown in \cite{Haase08} that the assumption of smoothness could be dropped in this case. 
However, it is well-known that, starting in dimension 3, this assumption cannot be dropped even in \cref{conj:oda}\ref{it:oda_weak}. 
Recent work in this area includes \cites{gubeladze2012convex, firla1999hilbert, Lundman_2012, Bogart_2015, haase2010generating}. We highlight that in \cite{beck19} it was shown that \cref{conj:oda}\ref{it:oda_weak} holds for 3-dimensional, centrally symmetric polytopes. 

In the special case of smooth combinatorial cubes a new \emph{polytopal} proof of \Cref{conj:oda}\ref{it:oda_weak} was given in \cite{curtis_cubes}. It is largely the methods of that paper that have inspired the present one.
\Cref{thm:absolute_ODA} proves \Cref{conj:oda}\ref{it:oda_strong} for smooth combinatorial products of simplices. In fact, it proves something stronger since \emph{both} $Q$ and $Q'$ are allowed to be weak Minkowski summands. The analogue of \Cref{thm:absolute_ODA} for arbitrary smooth polytopes fails even in dimension 2.
We note that a third proof of \Cref{thm:absolute_ODA} is available which bypasses \Cref{cor:unimodulartriangulation}
and sticks more closely to the ideas in \cite{curtis_cubes}, but which the present authors have omitted out of considerations of space.

\subsection*{Structure of the paper} 
In \Cref{sec:notation} we recall some basic facts about polytopes and define Cayley sums and Cayley towers.
In \Cref{sec:structure_thm}, we prove \Cref{thm:2faces}\ref{thm:combprodsimp}
and the very similar \Cref{thm:struct_thm_smooth}. 
In \Cref{sec:WMS_of_SCPS}, we introduce the all-important chimney construction, a minor variation of the one in \cite{haase21}*{\textsection~2.2.1}, and use it to prove  \Cref{thm:Nakajima}. 
\Cref{cor:unimodulartriangulation} is then immediate. We conclude the section by proving \Cref{thm:absolute_ODA}. 
In \Cref{sec:literature}, we interpret our results in the language of toric varieties and discuss their relationship to the existing literature.

\subsection*{AI disclosure} Anthropic's Claude (Fable 5 and 5.1) was used for proofreading. The writing and ideas are due entirely to the authors. 

\section{Preliminaries} \label{sec:notation}
    \subsection{Polytopes} 

    We begin by recalling some basic definitions and properties. Most of the basics we discuss here and much more may be found in \cite{ziegler95}.
    We write $e_1, \dots, e_n$ for the standard basis of $\mathbb{R}^n$ and $x_1, \dots, x_{n}$ for its dual basis (of linear functionals).
    A \define{polytope} $P$ in $\R^n$ is any set of the form $\Conv(S)$ where $S \subset \R^n$ is a finite set and $\Conv$ is the convex hull operator. 
    We say that $P$ is a \define{lattice polytope}
    if $S$ may be chosen to lie in the lattice $\Z^n \subset \R^n$. 
    The \define{dimension} of a polytope $P$, denoted $\dim P$, is the dimension of its affine hull $\Aff P$ (the smallest affine subspace containing it).
    If this affine subspace is all of $\R^n$,
    then $P$ is \define{full-dimensional}.

    A \define{supporting hyperplane} for a polytope $P$ is an affine hyperplane which intersects $P$ and such that $P$ lies entirely in one of the two closed halfspaces it delimits.
    A \define{face} of $P$ is the intersection of $P$ with a supporting hyperplane. We also admit $P$ itself as a face, and each face of $P$ is again a polytope.
    A face of $P$ not equal to $P$ is a \define{proper} face. 
    A maximal proper face is a \define{facet}.
    The $0$-dimensional faces of $P$ are its \define{vertices}.
    The $1$-dimensional faces are its \define{edges}.
    If $P$ has dimension $d$, each of its vertices lies in at least $d$ edges. Polytopes which achieve this minimum at every vertex are \define{simple}. 

    Recall that an integer vector is \define{primitive} if its entries are relatively prime.
    If $P$ is a lattice polytope in $\R^n$ and $p, q$ are adjacent vertices in its $1$-skeleton, the \define{primitive edge direction} from $p$ to $q$ is the unique primitive integer vector $v \in \Z^n$ such that 
    $q - p = mv$ for some positive integer $m$.
    The primitive edge directions at $p$ are the primitive edge directions from $p$ to $q$ as $q$ ranges over the neighbors of $p$ in the 1-skeleton of $P$. 
    We say that $P$ is \define{smooth} precisely when these integer vectors form a basis of a saturated sublattice of $\Z^n$, no matter the choice of $p$. 
    The nomenclature is inherited from toric geometry and reflects the fact that a lattice polytope is smooth if and only if the projective variety it defines is smooth. Any smooth polytope is necessarily simple.

There are various notions of isomorphism of polytopes which are relevant in different contexts. The ones we employ in this paper are the ones we now enumerate.
Let $P$ and $Q$ be polytopes in $\R^n$. We say that $P$ and $Q$ are \define{combinatorially isomorphic} when they have isomorphic face posets; an isomorphism between their face posets is a \define{combinatorial isomorphism}.
$P$ and $Q$ are \define{affinely isomorphic} if $Q$ is the image of $P$ under an affine isomorphism: an invertible affine-linear transformation $\R^n \to \R^n$.
They are \define{isometric} if there is a Euclidean motion (i.e.\ a distance-preserving affine transformation) taking one bijectively onto the other.
If $P$ and $Q$ are lattice polytopes, then we say they are \define{lattice isomorphic} if there is a lattice-preserving affine isomorphism taking one bijectively onto the other. 
Here, an affine isomorphism 
$f: \R^n \to \R^n$ is \define{lattice-preserving} if $f(\Z^n) = \Z^n$.
By identifying $\R^{n}$ in the usual way as a subspace of $\R^{m}$ for $n < m$
(and the lattice $\Z^n$ with a sublattice of $\Z^m$), we extend these notions to polytopes living in different ambient dimensions. Evidently, isometry and lattice isomorphism of polytopes each imply affine isomorphism, which in turn implies combinatorial isomorphism.

Given polytopes $P$ in $\mathbb{R}^n$ and $Q$ in $\mathbb{R}^m$, the product set $P \times Q \subseteq\mathbb{R}^{n+m}$ is again a polytope.
The faces of $P \times Q$ are the polytopes $F \times G$ where $F, G$ range over the faces of $P, Q$ respectively.
From this one deduces that the combinatorial (resp. affine, lattice, isometry) isomorphism class of $P \times Q$ depends only on the  combinatorial (resp.\ affine, lattice, \ isometry) isomorphism classes of $P$ and $Q$.
As an example which will be relevant for us, a \define{combinatorial cube} is any polytope combinatorially isomorphic to the \define{standard cube} $[0,1]^n$ for some $n$.
The \define{standard $n$-simplex} is the full-dimensional lattice polytope in $\R^n$ whose vertices are the origin and the $n$ standard basis vectors; any polytope combinatorially isomorphic to it is an \define{$n$-simplex}.
A \define{combinatorial product of simplices} is any polytope combinatorially isomorphic to some product of simplices. Since a line segment is a 1-simplex, every combinatorial cube is a combinatorial product of simplices.

\subsection{Normal fans}
Much of the background we discuss in this section may be found in \cite{oda88}*{Appendix}.
The \define{positive} or \define{conical hull} of a subset $S \subseteq \R^n$ is the set of all linear combinations with nonnegative coefficients of points in $S$. 
A \define{polyhedral cone} in $\R^n$ is the positive hull of a finite set of points.
We shall henceforth use \define{cone} as an abbreviation for \emph{polyhedral} cone.
The \define{dimension} of a cone is the dimension of its linear span. A cone in $\R^n$ of dimension $n$ is \define{full-dimensional}.
A cone is \define{rational} if it is the positive hull of a finite subset of the lattice $\Z^n$.

    A \define{supporting hyperplane} for a cone $C$ is a linear hyperplane such that $C$ lies entirely in one of the two closed halfspaces it delimits.
    A \define{face} of a cone $C$ is the intersection of $C$ with a supporting hyperplane. We also admit $C$ itself as a face. Each face of $C$ is again a cone. 
    The unique minimal face of $C$ is a linear subspace of $\R^n$ called its \define{lineality space} and denoted $\lin(C)$.
    A cone $C$ is \define{strictly convex} if $\lin(C) = \{0\}$. 
    
    A \define{ray} of a strictly convex cone $C$ is a 1-dimensional face.
    Every strictly convex cone is the convex hull of its rays and has at least as many rays as its dimension.
    A cone is \define{simplicial} if it achieves this minimum.
    If $C$ is strictly convex and rational, each of its rays is the positive hull of some integer vector $v$. The unique primitive such $v$ is called the \define{primitive ray generator}. 
    A rational, strictly convex cone is \define{smooth} if its primitive ray generators form a basis of a saturated sublattice of $\Z^n$. Any such cone is necessarily simplicial. 
    %More generally, we say that a rational polyhedral cone $C$ is smooth if its image in $\R^n/\lin(C)$ is smooth (after making the identifications explained above).

    A \define{fan} in $\R^n$ is a finite collection $\Sigma$ of polyhedral cones such that (i) whenever it contains a cone, it also contains all of its faces and (ii) if it contains any two cones $C$ and $D$, then the cone $C \cap D$ is a face of both $C$ and $D$. 
    A fan is \define{complete} if the setwise union of its cones is all of $\R^n$.
    A fan is \define{rational} (resp.\ \define{strictly convex}, \define{simplicial}, \define{smooth}) if each of its cones is rational (resp.\ {strictly convex}, {simplicial}, {smooth}). 
    Just as with polytopes, one defines \define{combinatorial}, \define{affine} and 
    \define{lattice isomorphism} of fans. In the last case, the fan is assumed to be rational.
    
    The \define{support function} of a polytope $P \subset \R^n$ is the function $h_P: \R^n \to \R$, given by
    \[
        h_P : v \mapsto \sup_{x \in P} \inner{v, x}.
    \]
    We can recover $P$ from $h_P$ via
    \[
        P = \bigcap_{v \in \R^n} \{x \in \R^n \mid \inner{v, x} \leq h_P(v)\}.
    \]
    These observations underlie the following basic fact:
    \begin{Fact}[cf.\ {\cite{oda88}*{Theorem A.18(a), Corollary A.19}}]\label{fact:support_criterion}
        The function $h_P$ is positively homogeneous,  convex, and piecewise linear on $\R^n$. Conversely, any function on $\R^n$ satisfying these properties has the form 
        $h_P$ for some polytope $P$ in $\R^n$.
    \end{Fact}
    \noindent
    Since $h_P$ is positively homogeneous and piecewise linear, its domains of linearity are full-dimensional cones in $\R^n$.  They are the maximal cones of a complete fan in $\R^n$, called the \emph{normal fan} of $P$, and denoted $\Sigma_P$.
    The cones of $\Sigma_P$ are in inclusion-reversing bijection with the faces of $P$.  
    Concretely, to a face $F$ of $P$, we associate the cone 
    \[
        \sigma^P_F \deq \left\{v \in \R^n \;\middle|\; 
        \text{$\inner{v, x} \leq \inner{v, y}$ whenever $x \in P$ and $y \in F$}\right\},
    \]
    called the \define{normal cone} of $F$.
    The domains of linearity of $h_P$ are the cones $\sigma_p^P$ as $p$ ranges over the vertices of $P$.
    These are the {maximal cones} of $\Sigma_P$. If $P$ is full-dimensional, the fan $\Sigma_P$ is strictly convex and the function $h_P$ is determined by its values on the rays of $\Sigma_P$, which are in bijection with the facets of $P$ under the correspondence described above.
    If $P$ is a lattice polytope, the fan $\Sigma_P$ is rational.
    If $P$ is full-dimensional then $\Sigma_P$ is strictly convex. If $P$ is smooth, then $\Sigma_P$ is smooth.

    Two polytopes in $\R^n$ with identical normal fans are called \define{Minkowski equivalent}.
    For instance, two translates of the same polytope are always Minkowski equivalent, as are two positive rescalings of one.
    For a simplex, \emph{all} Minkowski equivalent polytopes are obtained by combining these two operations. In particular, Minkowski equivalent simplices are affinely isomorphic. The same is not true for a general polytope.
    \begin{Fact}\label[Fact]{fact:smooth_simplex}
        Let $c_1, \dots, c_n$ be positive integers. The simplex $\Delta = \Conv(0, c_1 e_1, \dots, c_n e_n)$
        in $\mathbb{R}^n$ is smooth if and only if $c_1 = \dots = c_n$.
    \end{Fact}
\begin{Fact}\label[Fact]{fact:MEQdescent}
    Let $P$ and $Q$ be Minkowski equivalent polytopes. For each 
    $v \in \R^n$, the faces $P_v$ of $P$ and 
    $Q_v$ of $Q$ on which $\inner{v,-}$ is maximized are Minkowski equivalent.
\end{Fact}
\begin{Prop}\label[Prop]{prop:MEQaffine}
    If $P$ and $Q$ are Minkowski equivalent polytopes in $\mathbb{R}^n$ and 
    $f: \mathbb{R}^n \to \mathbb{R}^m$ is a linear map, then $f(P)$ and $f(Q)$ are Minkowski equivalent polytopes.
\end{Prop}
\begin{proof}
    Write $f^*: \mathbb{R}^m \to \mathbb{R}^n$ for the dual map. Then 
    $h_{f(P)} = h_P \circ f^*$ shows that the domains of linearity of $h_{f(P)}$ are the preimages under $f^*$ of the domains of linearity of $h_P$. Since the same statement holds for $Q$, the claim follows.
\end{proof}

\subsection{Minkowski sums and weak Minkowski summands}
The \define{Minkowski sum} of two polytopes $P$ and $Q$ in $\R^n$ is the set 
\[
    P + Q \, \deq \, \{x+y \mid x \in P, \; y \in Q\}.
\]
The set $P + Q$ is again a polytope. In fact, it is the convex hull of the points $p + q$ as $p, q$ range over the vertices of $P, Q$ respectively. The support function of $P+Q$ is given by $h_{P + Q} = h_P + h_Q$.
Evidently, if both $P$ and $Q$ are lattice polytopes, then so is $P+Q$.

\begin{Prop}\label{prop:nef}
    Let $P, Q$ be polytopes in $\R^n$. The following are equivalent:
    \begin{enumerate}[(1)]
        \item $h_Q$ is adapted to $\Sigma_P$, i.e.\ $h_Q$ is linear on each cone of $\Sigma_P$.\label{it:adapted}
        \item $\Sigma_Q$ coarsens $\Sigma_P$, that is, each cone of $\Sigma_P$ is contained in a cone of $\Sigma_Q$.\label{it:coarsen}
        \item $P + Q$ is Minkowski equivalent to $P$.\label{it:dilute}
    \end{enumerate}
    We say that $Q$ is a \define{weak Minkowski summand} of $P$ if these equivalent conditions hold.
\end{Prop}
\begin{proof}
    If \ref{it:adapted} holds, then each domain of linearity of $h_P$ is contained in a domain of linearity of $h_Q$. This shows \ref{it:adapted}$\Rightarrow$\ref{it:coarsen}.
    If \ref{it:coarsen} holds, then the common refinement of $\Sigma_P$ and $\Sigma_Q$ is $\Sigma_P$. Hence $\Sigma_{P+Q} = \Sigma_P$ by \cite{ziegler95}*{Proposition 7.12}. This shows \ref{it:coarsen}$\Rightarrow$\ref{it:dilute}.
    Finally, suppose \ref{it:dilute} holds.
    On each maximal cone of $\Sigma_P$, both $h_P$ and 
    $h_{P+Q}$ are linear. Hence so is 
    $h_Q = h_{P+Q} - h_{P}$. This shows 
    \ref{it:dilute}$\Rightarrow$\ref{it:adapted}.
\end{proof}
\noindent
It is an easy consequence of the definition that the collection of weak Minkowski summands of a polytope $P$ is closed under Minkowski sum.
Also, a weak Minkowski summand of a weak Minkowksi summand of $P$ is a weak Minkowksi summand of $P$.
By \Cref{prop:nef}\ref{it:dilute}, two polytopes are Minkowski equivalent if and only if each is a weak Minkowski summand of the other. From this and \cite{grunbaum}*{\textsection~15.1.2}, one deduces:
\begin{Fact}\label[Fact]{fact:pos_parallel}
    Let $P, P'$ be convex polytopes in $\R^n$.
    Then $P$ and $P'$ are Minkowski equivalent if and only if there is a combinatorial isomorphism $P \simeq P'$ under which corresponding edges
    $pq$ and $p'q'$ are \emph{positively} parallel.
    (Here $p'$ is the vertex in $P'$ corresponding to a vertex $p$ in $P$.)
\end{Fact}
\noindent
Here and elsewhere, (directed) edges $pq$ and $p'q'$ are \define{positively parallel} if the vector $q' - p'$ is a positive scalar multiple of $q-p$.

\begin{Prop}\label[Prop]{prop:WMSaffine}
    Suppose $P$ and $Q$ are polytopes in $\mathbb{R}^n$ with $Q$ a weak Minkowski summand of $P$. 
    If $f: \mathbb{R}^n \to \mathbb{R}^m$ is a linear map, then $f(Q)$ is a weak Minkowski summand of $f(P)$.
\end{Prop}
\begin{proof}
    By \Cref{prop:nef}\ref{it:dilute}, $P+Q$ is Minkowski equivalent to $P$. By \Cref{prop:MEQaffine}, $f(P+Q)$ is Minkowski equivalent to $f(P)$. Since $f$ is linear, $f(P+Q) = f(P) + f(Q)$. Thus by \Cref{prop:nef}\ref{it:dilute}, $f(Q)$ is a weak Minkowski summand of $f(P)$.
\end{proof}
\noindent
The following is a convenient way to show that one polytope is a weak Minkowski summand of another.
\begin{Lm}\label{lm:Minkowski_limit}
    Let $P$ be a polytope in $\R^n$.
    Let $Q_1 \supseteq Q_2 \supseteq \dots$ be a nested sequence of polytopes Minkowski equivalent to $P$. Then 
    $Q \deq \bigcap_i Q_i$ is a weak Minkowski summand of $P$.
\end{Lm}
\begin{proof}
    Observe that the inequalities 
    $h_{Q_1} \geq h_{Q_2} \geq \dots$ hold pointwise.
    Let $h_{\infty} = \inf_{i} h_{Q_i}$. Since each function $h_{Q_i}$ is positively homogeneous, convex, and linear on each maximal cone of $\Sigma_P$, the same is true of $h_{\infty}$.
    By \Cref{fact:support_criterion}, $h_{\infty}$ is the support function of some polytope $Q_{\infty}$. By \Cref{prop:nef}, $Q_{\infty}$ is a weak Minkowski summand of $P$. For each $i$, the inequality $h_{Q_{\infty}} \leq h_{Q_i}$ implies $Q_{\infty} \subseteq Q_i$. 
    Hence $Q_{\infty} \subseteq Q$. In the opposite direction, for each $i$, $Q \subseteq Q_i$ implies
    $h_{Q} \leq h_{Q_i}$. Thus $h_{Q} \leq \inf_i h_{Q_i} = h_{Q_{\infty}}$. So $Q \subseteq Q_{\infty}$.
\end{proof}

\subsection{Cayley sums}\label{subsec:SDP}
\begin{Df}
Let $P_0, \dots, P_m$ be polytopes in $\R^n$.
Let $k$ be a positive real number.
The \define{$k$-Cayley sum} of $P_0, ..., P_m \subseteq \R^n$ is the polytope 
\[
    \Cay_k (P_0, \dots, P_m) \deq \Conv \big((P_0 \times \{s_0\}) \, \cup\, (P_1 \times \{s_1\}) \, \cup \, \dots \, \cup\,  (P_m \times \{s_m\}) \big) \subseteq \R^n \times \R^m
\]
where $s_0 = 0$ and $s_i = k e_i$ ($1 \leq i \leq m$) are the vertices of the standard $m$-simplex rescaled by a factor of $k$.
By a \define{Cayley sum} of $P_0, \dots, P_m$, we mean this polytope for some choice of $k$.\footnote{This is slightly more permissive than the prevailing convention.} 
When $k$ is an integer and $P_0, \dots, P_m$ are \emph{lattice}, we call the resulting polytope an \define{integral Cayley sum}.
We may omit the parameter $k$ in the notation, in which case the intended value is $1$.
We say that the Cayley sum
$\Cay_k(P_0, \dots, P_m)$ is \define{coherent} if $P_0, \dots, P_m$ are Minkowski equivalent polytopes.\footnote{In \cite{dickenstein}, these are called \emph{strict} Cayley sums. However, this would conflict with our nomenclature for chimneys in \Cref{sec:WMS_of_SCPS}.}
\end{Df}
\noindent
We note that an integral Cayley sum is always a lattice polytope.
We remark also that varying the parameter $k$ results in polytopes that are both Minkowski equivalent and affinely isomorphic.
The choice of $k$ becomes more significant when we are considering polytopes up to lattice isomorphism. 
\begin{Obs}\label[Obs]{obs:nested_cayley}
    Cayley sums can be defined recursively in terms of binary Cayley sums:
    \[
        \Cay_k(P_0, P_1, \dots, P_m) = 
        \Cay_k(\Cay_k(P_0, P_1, \dots, P_{m-1}), P_m \times \{0\}^{m-1}).
    \]
    
\end{Obs}
\begin{Fact}\label[Fact]{fact:cayley=combprod}
A coherent Cayley sum $\Cay_k(P_0, \dots, P_m)$ is combinatorially isomorphic to the product $P_0 \times \Delta$, where $\Delta$ is an $m$-simplex.
\end{Fact}

\begin{Prop}\label[Prop]{prop:MEQ2Cayley}
    If $Q = \Cay_{k}(P_0, \dots, P_m)$ is a coherent Cayley sum of polytopes in $\mathbb{R}^n$ and $Q'$ is Minkowski equivalent to $Q$, then, up to translation, $Q'$ takes the form $\Cay_{k'}(P_0', \dots, P_m')$
    for polytopes $P_0', \dots, P_m'$ Minkowski equivalent to $P_0$ and some positive real number $k'$.
\end{Prop}
\begin{proof}
    As $P_0, \dots, P_m$ are parallel faces of $Q$, we have corresponding parallel faces 
    $P_0', \dots, P_m'$ of $Q'$ Minkowski equivalent to them by \Cref{fact:MEQdescent}. If we choose corresponding vertices 
    $p_0, \dots, p_m$ and $p_0', \dots, p_m'$
    in $P_0, \dots, P_m$ and $P_0', \dots, P_m'$ respectively, then
    another application of \Cref{fact:MEQdescent} shows that
    the simplices 
    $\Delta \deq \Conv(p_0, \dots, p_m)$
    and $\Delta' \deq \Conv(p_0', \dots, p_m')$ are Minkowski equivalent. Writing 
    $\pi: \mathbb{R}^n \times \mathbb{R}^m \to \mathbb{R}^m$ for the second projection, the images $\pi(\Delta)$ and $\pi(\Delta')$ of these simplices are Minkowski equivalent by \Cref{prop:MEQaffine}. Since $\pi(\Delta)$ is a translated positive rescaling of the standard simplex in $\R^m$, so is $\pi(\Delta')$.
    This implies that $Q'$ is a Cayley sum of $P_0', \dots, P_m'$.
\end{proof}
\noindent
With these facts in mind, we make the following definition.
\begin{Df}\label[Df]{df:Cayley_tower}
    A polytope is a \define{Cayley tower} if it either consists of a single point or is a coherent Cayley sum of the form 
    $\Cay_{k}(P_0, \dots, P_m)$ for some positive real number $k$ where $P_0, \dots, P_m$ are themselves translated (Minkowski equivalent) Cayley towers.

    A polytope is an \define{integral Cayley tower} if it either consists of a single \emph{lattice} point or is a coherent Cayley sum of the form 
    $\Cay_{k}(P_0, \dots, P_m)$ for some positive \emph{integer} $k$ where $P_0, \dots, P_m$ are themselves translated (Minkowski equivalent) \emph{integral} Cayley towers.
\end{Df}
\noindent
It follows immediately from the definition and 
\Cref{fact:cayley=combprod} that Cayley towers are combinatorially isomorphic to products of simplices. Similarly, from \Cref{prop:MEQ2Cayley} it follows that being a Cayley tower up to translation is a property invariant under Minkowski equivalence.
An easy inductive argument shows:
\begin{Obs}\label[Obs]{obs:integral=lattice}
    A Cayley tower is an \emph{integral} Cayley tower if and only if it is a lattice polytope.
\end{Obs}

\subsection{Triangulations}
Let $P$ be a polytope in $\R^n$. A \define{triangulation} of $P$ is a finite collection $\c T$ of simplices satisfying:
\begin{itemize}
    \item The union $\bigcup_{T \in \c T} T$ is equal to $P$ as a set.
    \item For any simplex $T \in \c T$, each face of $T$ is again in $\c T$.
    \item For any two simplices $T, T' \in \c T$, the intersection $T \cap T'$ is a face of both $T$ and $T'$. 
\end{itemize}
The simplices of $\c T$ are its \define{faces}.
The \define{vertices} of $\c T$ are the 0-dimensional simplices in $\c T$. A nonempty collection $V$ of vertices of $\c T$ is a \define{non-face} if there is no simplex in $\c T$ whose set of vertices is exactly $V$. The triangulation $\c T$ is \define{flag} if its minimal non-faces have cardinality 2.

A triangulation $\c T$ as above is \define{regular} if there exists a real-valued convex function $f$ with domain $P$
whose domains of linearity are 
precisely the maximal simplices of $\mathcal{T}$.

A \define{unimodular} simplex is one which is lattice isomorphic to a standard simplex.
When $P$ is a lattice polytope, we say that a triangulation $\c T$ of $P$ is \define{unimodular} provided each simplex in $\c T$ is unimodular. It is \define{quadratic} if it is simultaneously unimodular, regular and flag.

\section{The structure theorem} \label{sec:structure_thm}

Our goal in this section is to prove Theorems \ref{thm:2faces}\ref{thm:combprodsimp} and \ref{thm:struct_thm_smooth}.
To this end, we fix a simple polytope $P$ whose two-dimensional faces consist only of triangles and trapezoids. By \Cref{thm:2faces}\ref{thm:combprod}, such a polytope must be a combinatorial product of simplices. We may therefore fix a combinatorial isomorphism 
$P \simeq \Delta_1\times \dots \times \Delta_r$, where each $\Delta_i$ is a simplex. 
We write $E(p)$ for the set of edges of $P$ incident to the vertex $p$ of $P$.
Given a vertex $\xi$ of some $\Delta_i$, we write $P_{\xi}$ for the face of $P$ corresponding to the subproduct
\[
    \Delta_1 \times \dots \times \Delta_{i - 1} \times \{\xi\} \times \Delta_{i + 1} \times \dots \times \Delta_r.
\]
\noindent
Our immediate aim is to prove the following result.
\begin{Prop}\label[Prop]{prop:parallel_main}
    There exists some $i \in [r]$ such that  the faces $P_{\xi}$ for vertices ${\xi} \in \Delta_i$ are all parallel.
\end{Prop}
\noindent
\Cref{prop:parallel_main} generalizes a result of the first author for combinatorial cubes \cite{curtis_cubes}*{Theorem 3.7}, which one recovers by imposing $\dim \Delta_i = 1$ for all $i$.
We begin with a couple of definitions.
\begin{Df}
    In the combinatorial product $P\simeq \Delta_1\times \dots \times \Delta_r$, an edge $\varepsilon$ of $P$ corresponds to a choice of an edge in some factor $\Delta_i$ and a vertex in all remaining factors. We then say that $\varepsilon$ has \define{type} $i$. 
\end{Df}
\noindent
Since $P$ is a simple polytope, for any edges $\varepsilon_1, \dots, \varepsilon_k$ in $P$ incident to a fixed vertex $p$, there is a unique face $F$ of $P$ containing $p$ such that the edges of $F$ incident to $p$ are exactly 
$\varepsilon_1, \dots, \varepsilon_k$. In such a case, we will say that $\varepsilon_1, \dots, \varepsilon_k$ \define{span} the face $F$.
Two adjacent edges in $P$ span a triangle if and only if they have the same type. Otherwise, they span a quadrilateral.

\begin{Df}\label[Df]{df:arrows}
Suppose $\varepsilon, \delta$ are edges of $P$ incident to a common vertex.
\begin{itemize}
    \item We write $\varepsilon \sim \delta$ if $\varepsilon$ and $\delta$ have the same type. We write $\varepsilon \nsim \delta$ otherwise.
    \item If $\varepsilon \nsim \delta$, consider the 2-face spanned by $\varepsilon$ and $\delta$.
    We write $\tau_{\varepsilon} \delta$ for the edge in this quadrilateral situated opposite $\delta$.
    % We call $\tau_{\varepsilon} \delta$ the \emph{translate} of $\delta$ along $\varepsilon$.
    Note that $\delta$ and $\tau_{\varepsilon} \delta$ have the same type and
    ${\tau_{\varepsilon} \tau_{\varepsilon} \delta = \delta}$.
    \item We write $\varepsilon \midarrow \delta$ to mean that $\varepsilon \nsim \delta$ and $\tau_{\delta} \varepsilon$ is parallel to $\varepsilon$.
    We also say that $\delta$ moves $\varepsilon$.
    \item We write $\varepsilon \midnarrow \delta$ to mean that $\varepsilon \nsim \delta$ and that $\tau_{\delta} \varepsilon$ is \emph{not} parallel to $\varepsilon$. We also say that $\delta$ does not move $\varepsilon$.
\end{itemize}
\end{Df}

\begin{Lm}\label[Lm]{lm:rules}
    The following hold for any edges $\varepsilon, \delta, \gamma$ in $P$ incident to a common vertex.
    \begin{enumerate}
        \item If $\varepsilon \midnarrow \delta$ then $\delta \midarrow \varepsilon$.\label{it:either_or}
        \item If $\varepsilon \midarrow \delta$ and $\delta \sim \gamma$ then $\varepsilon \midarrow \gamma$.\label{it:equiv}
        \item If $\varepsilon \midarrow \delta$
        then $\tau_{\delta} \varepsilon \midarrow \delta$.\label{it:symmetric}
        \item If $\varepsilon \midarrow \delta$,
        $\varepsilon \midarrow \gamma$ and $\gamma \nsim \delta$
        then $\tau_\gamma \varepsilon \midarrow \tau_{\gamma} \delta$.\label{it:cube1}
        \item If $\varepsilon \midarrow \delta$, 
        $\varepsilon \midnarrow \gamma$ and $\gamma \midnarrow \delta$ then $\tau_{\gamma} \varepsilon \midnarrow \tau_{\gamma} \delta$.\label{it:cube2}
    \end{enumerate}
\end{Lm}
\begin{proof}
    We shall see that each of the above claims is a relatively simple statement about a polytope combinatorially isomorphic to a square, a triangular prism or a three-dimensional cube.
    Note that every face of $P$ is a combinatorial product of simplices.
    \begin{enumerate}
        \item The assumption by definition implies $\varepsilon \nsim \delta$. So $\varepsilon$ and $\delta$ form a quadrilateral 2-face of $P$, which must be a trapezoid by hypothesis.
        By assumption, $\varepsilon$ is not parallel to its opposite edge in this quadrilateral. Hence $\delta$ must be parallel to its opposite edge.
        \item The assumptions $\varepsilon \nsim \delta$ and $\delta \sim \gamma$ imply that the 3-face of $P$ spanned by $\varepsilon, \delta, \gamma$ is combinatorially isomorphic to a triangular prism, as depicted in \Cref{fig:rule_b}.
        \begin{figure}[ht]
            \centering
            \def\svgwidth{0.4\textwidth}
            %% Creator: Inkscape 1.4.2 (ebf0e940, 2025-05-08), www.inkscape.org
%% Accompanies image file 'prism_img.pdf' (pdf, eps, ps)
\begingroup%
  \makeatletter%
  \providecommand\color[2][]{%
    \errmessage{(Inkscape) Color is used for the text in Inkscape, but the package 'color.sty' is not loaded}%
    \renewcommand\color[2][]{}%
  }%
  \providecommand\transparent[1]{%
    \errmessage{(Inkscape) Transparency is used (non-zero) for the text in Inkscape, but the package 'transparent.sty' is not loaded}%
    \renewcommand\transparent[1]{}%
  }%
  \providecommand\rotatebox[2]{#2}%
  \newcommand*\fsize{\dimexpr\f@size pt\relax}%
  \newcommand*\lineheight[1]{\fontsize{\fsize}{#1\fsize}\selectfont}%
  \ifx\svgwidth\undefined%
    \setlength{\unitlength}{779.62211068bp}%
    \ifx\svgscale\undefined%
      \relax%
    \else%
      \setlength{\unitlength}{\unitlength * \real{\svgscale}}%
    \fi%
  \else%
    \setlength{\unitlength}{\svgwidth}%
  \fi%
  \global\let\svgwidth\undefined%
  \global\let\svgscale\undefined%
  \makeatother%
  \begin{picture}(1,0.7183633)%
    \lineheight{1}%
    \setlength\tabcolsep{0pt}%
    \put(0,0){\includegraphics[width=\unitlength,page=1]{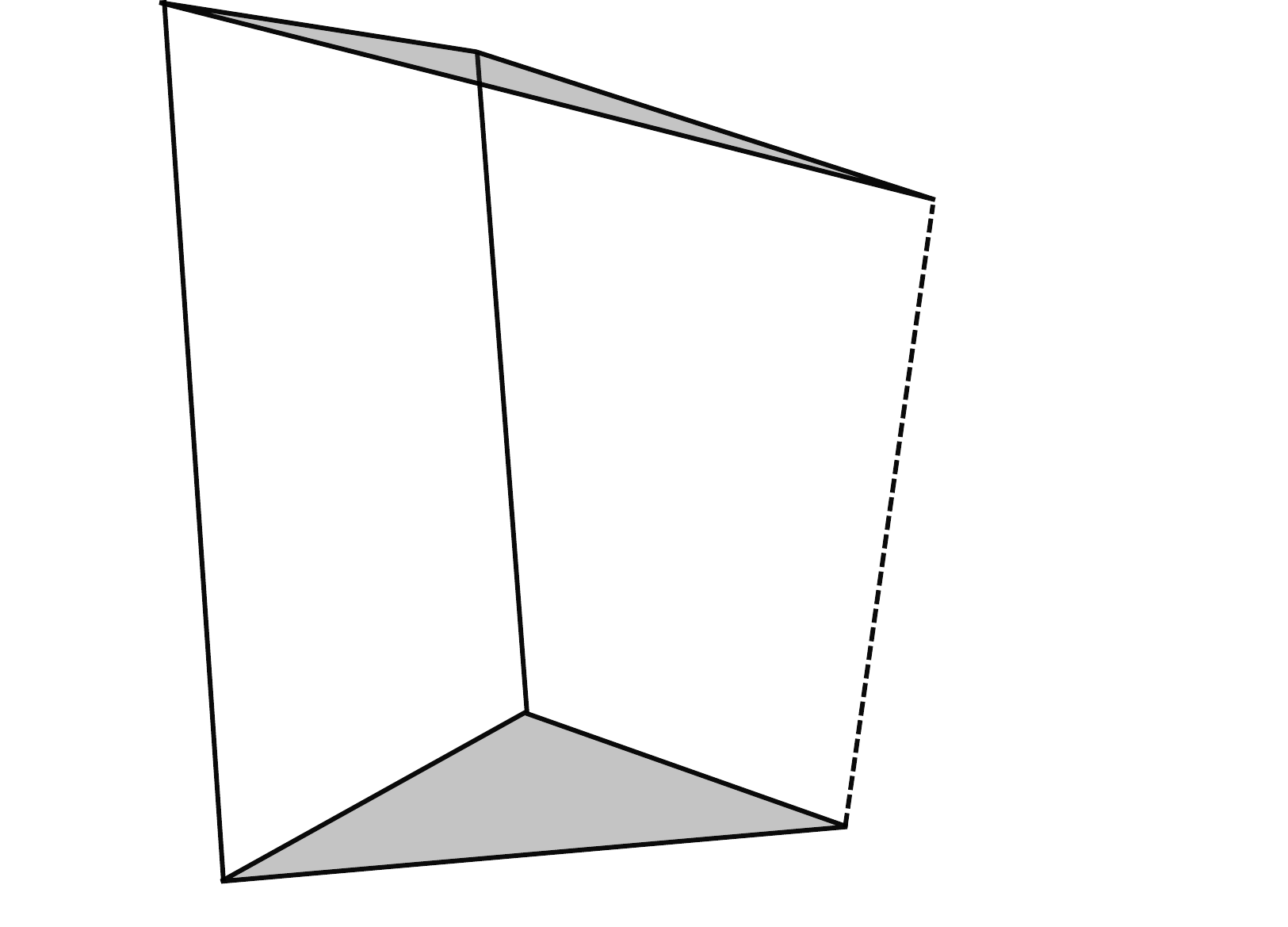}}%
    \put(0.16296137,0.34634879){\color[rgb]{0,0,0}\transparent{0.95899999}\makebox(0,0)[lt]{\lineheight{1.25}\smash{\begin{tabular}[t]{l}$\varepsilon$\end{tabular}}}}%
    \put(0.27017142,0.12093362){\color[rgb]{0,0,0}\transparent{0.95899999}\makebox(0,0)[lt]{\lineheight{1.25}\smash{\begin{tabular}[t]{l}$\delta$\end{tabular}}}}%
    \put(0.42852061,0.01665389){\color[rgb]{0,0,0}\transparent{0.95899999}\makebox(0,0)[lt]{\lineheight{1.25}\smash{\begin{tabular}[t]{l}$\gamma$\end{tabular}}}}%
    \put(0.31756299,0.41114689){\color[rgb]{0,0,0}\transparent{0.95899999}\makebox(0,0)[lt]{\lineheight{1.25}\smash{\begin{tabular}[t]{l}$\tau_{\delta} \varepsilon$\end{tabular}}}}%
    \put(0.60385451,0.28709383){\color[rgb]{0,0,0}\transparent{0.95899999}\makebox(0,0)[lt]{\lineheight{1.25}\smash{\begin{tabular}[t]{l}$\tau_{\gamma} \varepsilon$\end{tabular}}}}%
    \put(0,0){\includegraphics[width=\unitlength,page=2]{images/prism_img.pdf}}%
  \end{picture}%
\endgroup%
            \caption{Depiction of rule (b).}
            \label{fig:rule_b}
        \end{figure}
        Let $L$, $L'$ and $L''$ be the affine lines spanned by 
        $\varepsilon$, $\tau_{\delta} \varepsilon$ and $\tau_{\gamma} \varepsilon$ respectively. By hypothesis, $L$ and $L'$ are parallel. We need to show that $L$ and $L''$ are parallel. Suppose not. Since $L$ and $L''$ are coplanar, they must intersect. Since $L'$ and $L''$ cannot be parallel but are coplanar, they must also intersect. Since $L \cap L' = \varnothing$, the line $L''$ intersects the plane spanned by $L$ and $L'$ in two distinct points. Hence $L$, $L'$ and $L''$ must be jointly coplanar, which is a contradiction. 
        \item The assumption implies $\varepsilon \nsim \delta$. 
        So there is a (unique) 2-face of $P$ containing the edges $\varepsilon$, $\delta$ and $\tau_{\delta} \varepsilon$. Both assertions now become the statement that $\varepsilon$ and $\tau_{\delta} \varepsilon$ are parallel.
        \item The assumptions imply that $\varepsilon, \delta, \gamma$ belong to three distinct types, so span a 3-face $C$ of $P$ which is a combinatorial cube.
        The assumptions are now that $\varepsilon$, $\tau_{\delta} \varepsilon$ 
        and $\tau_{\gamma} \varepsilon$ are all parallel.
        Let $H$ be the affine plane spanned by the parallel edges $\tau_{\gamma} \varepsilon$ and $\tau_{\delta} \varepsilon$. Then $H$ divides $C$ into two (combinatorial) triangular prisms in which each quadrilateral facet is a trapezoid. Applying the argument in \ref{it:equiv} to these prisms completes the proof; see \Cref{fig:d}.
        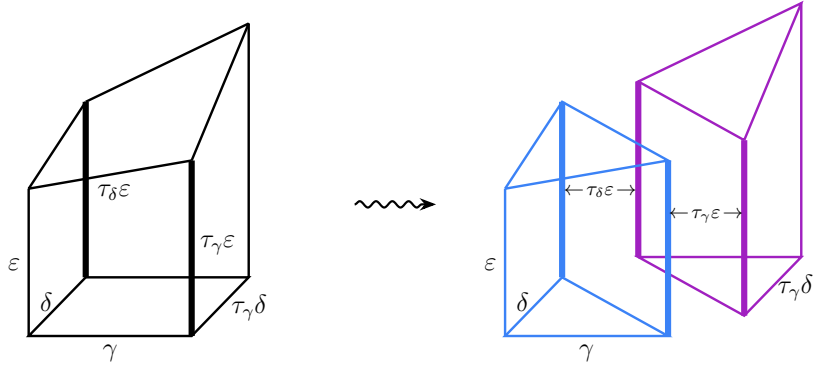
\begin{figure}[H]
        % \centering
        \resizebox{0.7\linewidth}{!}{
            \definecolor{sketchblue}{HTML}{3B82F6}
\definecolor{sketchpurple}{HTML}{A020C0}

\begin{tikzpicture}[line width=1.1pt,
    % oblique projection: x right, y into the page (up-right), z up
    x={(1cm,0cm)}, y={(0.35cm,0.36cm)}, z={(0cm,1cm)},
    cut edge/.style={line width=2.8pt}]

  \def\s{2.6}    % base side (lengths of gamma and delta)
  \def\hO{2.35}  % height eps at the origin
  \def\hD{2.8}   % height tau_delta eps at delta
  \def\hG{2.8}   % height tau_gamma eps at gamma
  \def\hGD{4.05} % height at gamma+delta
  \def\dd{0.9}   % separation of the two prisms along (1,1,0)

  % --- The unsplit solid ---
  \begin{scope}
    % base
    \draw (0,0,0) -- (\s,0,0) -- (\s,\s,0) -- (0,\s,0) -- cycle;
    % vertical edges (the two cut edges thick)
    \draw (0,0,0) -- (0,0,\hO);
    \draw[cut edge] (\s,0,0) -- (\s,0,\hG);
    \draw[cut edge] (0,\s,0) -- (0,\s,\hD);
    \draw (\s,\s,0) -- (\s,\s,\hGD);
    % top edges
    \draw (0,0,\hO) -- (\s,0,\hG) (0,0,\hO) -- (0,\s,\hD);
    \draw (\s,\s,\hGD) -- (\s,0,\hG) (\s,\s,\hGD) -- (0,\s,\hD);
    % labels
    \node[left]  at (0,0,1.15)      {\Large $\varepsilon$};
    \node[below] at (1.3,0,0)       {\Large $\gamma$};
    \node        at (-0.35,1.8,-0.1) {\Large $\delta$};
    \node[right] at (0.05,\s,1.35)  {\Large $\tau_\delta\varepsilon$};
    \node[right] at (\s,0,1.45)     {\Large $\tau_\gamma\varepsilon$};
  \end{scope}
  \node[right] at ($(\s+0.5,0,0.45)$)
      {\Large $\tau_\gamma\delta$};

  % --- Squiggly arrow ---
  \draw[-Stealth, line width=1pt, decorate,
        decoration={snake, amplitude=0.5mm, segment length=2.6mm, post length=1.6mm}]
    (5.2,0,2.1) -- (6.5,0,2.1);

  % --- The two triangular prisms ---
  \begin{scope}[shift={(7.6,0,0)}]

    % Purple prism over the triangle gamma, delta, gamma+delta
    \begin{scope}[sketchpurple, shift={(\dd,\dd,0)}]
      \draw (\s,0,0) -- (\s,\s,0) -- (0,\s,0) -- cycle;              % bottom
      \draw (\s,0,\hG) -- (\s,\s,\hGD) -- (0,\s,\hD) -- (\s,0,\hG);  % top
      \draw[cut edge] (\s,0,0) -- (\s,0,\hG);
      \draw[cut edge] (0,\s,0) -- (0,\s,\hD);
      \draw (\s,\s,0) -- (\s,\s,\hGD);
    \end{scope}
    \node[right] at ($(\s,0,0.45)+(\dd + 0.4,\dd,0)$)
      {\Large $\tau_\gamma\delta$};

       % Blue prism over the triangle 0, gamma, delta
    \begin{scope}[sketchblue]
      \draw (0,0,0) -- (\s,0,0) -- (0,\s,0) -- cycle;                % bottom
      \draw (0,0,\hO) -- (\s,0,\hG) -- (0,\s,\hD) -- (0,0,\hO);      % top
      \draw (0,0,0) -- (0,0,\hO);
      \draw[cut edge] (\s,0,0) -- (\s,0,\hG);
      \draw[cut edge] (0,\s,0) -- (0,\s,\hD);
    \end{scope}
    \node[left]  at (0,0,1.15)      {\Large $\varepsilon$};
    \node[below] at (1.3,0,0)       {\Large $\gamma$};
    \node        at (-0.35,1.8,-0.1) {\Large $\delta$};

    % identification labels between the matching cut edges
    \node at ($(0,\s,1.2)!0.5!(\dd,\s+\dd,1.2)$)
      {\small $\leftarrow\!\tau_\delta\varepsilon\!\rightarrow$};
    \node at ($(\s,0,1.75)!0.5!(\s+\dd,\dd,1.75)$)
      {\small $\leftarrow\!\tau_\gamma\varepsilon\!\rightarrow$};
  \end{scope}
\end{tikzpicture}
        }
        \caption{Depiction of rule (d).}
        \label{fig:d}
        \end{figure}
        
        \item As in \ref{it:cube1}, the assumptions imply that $\varepsilon, \delta, \gamma$ span a 3-face of $P$ combinatorially isomorphic to a cube. Arguing in the contrapositive, we will show that 
        $\varepsilon \midarrow \delta$, $\varepsilon \midnarrow \gamma$
        and $\tau_\gamma \varepsilon \midarrow \tau_{\gamma} \delta$ together imply $\gamma \midarrow \delta$. 
        Let $\Pi$ be the affine plane spanned by $\varepsilon$ and $\gamma$ and let $\tau_{\delta} \Pi$ be the plane spanned by $\tau_{\delta} \varepsilon$ and $\tau_{\delta} \gamma$.
        The first assumption means that $\varepsilon$ is parallel to $\tau_{\delta} \varepsilon$.
        The third assumption means that $\tau_{\gamma} \varepsilon$
        is parallel to $\tau_{\tau_{\gamma} \delta} \tau_{\gamma} \varepsilon$.
        The second assumption means that $\varepsilon$ and $\tau_{\gamma} \varepsilon$ are \emph{not} parallel. 
        Thus the non-parallel lines spanned by $\varepsilon$ and $\tau_{\gamma} \varepsilon$ on the plane $\Pi$ are each parallel to the plane $\tau_{\delta} \Pi$. Hence $\Pi$ and $\tau_{\delta} \Pi$ are parallel. 
        In particular, the intersections of $\Pi$ and $\tau_{\delta} \Pi$ with the plane spanned by $\gamma$ and $\delta$ are parallel. That is, $\gamma$ and $\tau_{\delta} \gamma$ are parallel.
        All this is illustrated in \Cref{fig:e}.
        \qedhere
        \end{enumerate}
\end{proof}

\begin{figure}[h!t]
        % \centering
        \resizebox{0.5\linewidth}{!}{
            \begin{tikzpicture}[line width=1.1pt]
  % bottom face vertices
  \coordinate (A) at (1.0,1.3);   % origin
  \coordinate (B) at (7.4,1.3);   % A + gamma
  \coordinate (C) at (2.7,2.8);   % A + delta
  \coordinate (D) at (9.1,2.8);   % C + tau_delta gamma = B + tau_gamma delta
  % top face vertices
  \coordinate (A') at (2.56,6.34); % A + eps
  \coordinate (B') at (7.5,5.0);   % B + tau_gamma eps
  \coordinate (C') at (4.44,8.04); % C + tau_delta eps
  \coordinate (D') at (9.1,6.5);   % D + tau_{tau_gamma delta} tau_gamma eps

  % bottom face
  \draw (A) -- (B) -- (D) -- (C) -- cycle;
  % vertical edges
  \draw (A) -- (A') (B) -- (B') (C) -- (C') (D) -- (D');
  % top face
  \draw (A') -- (B') -- (D') -- (C') -- cycle;

  % edge labels
  \node at (3.0,0.85)  {\Large $\gamma$};
  \node at (1.75,2.35) {\Large $\delta$};
  \node at (4.8,2.45)  {\Large $\tau_\delta\gamma$};
  \node at (8.75,1.75) {\Large $\tau_\gamma\delta$};
  \node at (1.35,3.9)  {\Large $\varepsilon$};
  \node at (3.75,4.6)  {\Large $\tau_\delta\varepsilon$};
  \node[right] at (7.5,4.15) {\Large $\tau_\gamma\varepsilon$};
  \node[right] at (9.2,4.65) {\Large $\tau_{\tau_\gamma\delta}\,\tau_\gamma\varepsilon$};
\end{tikzpicture}
        }
        \caption{Depiction of rule (e).}
        \label{fig:e}
        \end{figure}
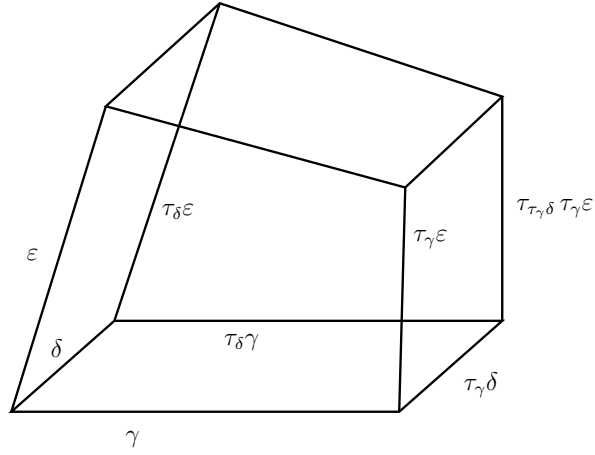
        
\begin{Lm}\label[Lm]{lm:vertex_acyclic}
    Let $p$ be any vertex of $P$. Consider the directed graph $G(p)$ with vertex set $E(p)$ and directed edges given by the $\midnarrow$ relation. Then the graph $G(p)$ has no directed cycles.
\end{Lm}
\begin{proof}
    Suppose for a contradiction that
    \[
        {\varepsilon}_1 \midnarrow {\varepsilon}_2 \midnarrow \dots \midnarrow {\varepsilon}_m
        \midnarrow {\varepsilon}_1
    \]
    is a directed cycle.
    Note that 
    $m \geq 3$ by \Cref{lm:rules}\ref{it:either_or}. We may assume that this cycle has minimal length among all cycles in all graphs $G(q)$ as $q$ ranges over the vertices of $P$. 
    If we had ${\varepsilon}_a \sim {\varepsilon}_b$ for some $1 \leq a < b \leq m$, then \Cref{lm:rules}\ref{it:equiv} would give us ${\varepsilon}_{b-1} \midnarrow \varepsilon_a$, producing a shorter cycle. It follows that ${\varepsilon}_1, \dots, {\varepsilon}_m$ are all of different type. 
    Furthermore, we must have ${\varepsilon}_a \midarrow \varepsilon_b$ whenever $b \not\equiv a, a+1 \pmod m$ since ${\varepsilon}_a \midnarrow {\varepsilon}_b$ would create a shorter cycle. 
    
    Let $p' \neq p$ be the other endpoint of the edge ${\varepsilon}_1$. Then 
    $\tau_{{\varepsilon}_1} {\varepsilon}_a$ for $1 < a \leq m$ are edges of distinct type in $E(p')$. Note that 
    $\tau_{{\varepsilon}_1} {\varepsilon}_a \midarrow \varepsilon_1$ whenever
    $1 < a < m$ by \Cref{lm:rules}\ref{it:symmetric}.
    We claim that $\tau_{{\varepsilon}_1} {\varepsilon}_a \midnarrow \tau_{{\varepsilon}_1} {\varepsilon}_{a + 1}$ if $1 < a < m$. If instead $\tau_{{\varepsilon}_1} {\varepsilon}_a \midarrow \tau_{{\varepsilon}_1} {\varepsilon}_{a + 1}$ then applying \Cref{lm:rules}\ref{it:cube1} at the vertex $p'$ gives 
    $\tau_{{\varepsilon}_1} \tau_{{\varepsilon}_1} {\varepsilon}_a \midarrow \tau_{{\varepsilon}_1} \tau_{{\varepsilon}_1} {\varepsilon}_{a + 1}$, i.e.\ 
    ${\varepsilon}_a \midarrow \varepsilon_{a + 1}$, which is a contradiction.

    Finally, since ${\varepsilon}_m \midarrow \varepsilon_2$,
    ${\varepsilon}_m \midnarrow {\varepsilon}_1$ and ${\varepsilon}_1 \midnarrow {\varepsilon}_2$, \Cref{lm:rules}\ref{it:cube2} gives 
    $\tau_{{\varepsilon}_1} {\varepsilon}_m \midnarrow \tau_{{\varepsilon}_1} {\varepsilon}_2$. This means that $G(p')$ contains the cycle
    \[
        \tau_{{\varepsilon}_1} {\varepsilon}_2 \midnarrow \tau_{{\varepsilon}_1} {\varepsilon}_3 \midnarrow \dots \midnarrow \tau_{{\varepsilon}_1} {\varepsilon}_m
        \midnarrow \tau_{{\varepsilon}_1} {\varepsilon}_2
    \]
    of length $m-1$. Since $m$ was assumed minimal, this is the desired contradiction.
\end{proof}

\begin{Lm}\label[Lm]{lm:type_acyclic}
    Fix a vertex $p$ of $P$.
    Then there exists a permutation $\pi: [r] \to [r]$ such that 
    if $\pi(i) < \pi(j)$ then
    $\varepsilon \midarrow \delta$ for all ${\varepsilon} \in E(p)$ of type $i$ and all $\delta \in E(p)$ of type $j$.
\end{Lm}
\begin{proof}
    We will (momentarily) extend the
    $\midnarrow$ notation to types. 
    Namely, for distinct $i, j \in [r]$ we write 
    $i \midnarrow j$ if $\varepsilon \midnarrow f$ for some ${\varepsilon} \in E(p)$ of type $i$ and some $f \in E(p)$ of type $j$. The claim reduces to the assertion that the $\midnarrow$ relation on types extends to a strict total order on $[r]$. This holds if and only if the graph on $[r]$ with edges given by $\midnarrow$ is directed acyclic. Suppose we had a directed cycle in this graph:
    \[
        i_0 \midnarrow i_1 \midnarrow i_2 \midnarrow \dots \midnarrow i_m = i_0
    \]
    For each $0 \leq a < m$, we can find edges $\varepsilon_a \in E(p)$ of type $i_a$ and 
    $\delta_a \in E(p)$ of type $i_{a + 1}$ such that 
    $\varepsilon_a \midnarrow \delta_a$. We let 
    $\varepsilon_m \deq \varepsilon_0$, $\delta_m \deq \delta_0$.
    For all $0 \leq a < m$, since $\delta_a \sim \varepsilon_{a + 1}$ we have 
    $\varepsilon_a \midnarrow {\varepsilon_{a+1}}$ by the contrapositive of \Cref{lm:rules}\ref{it:equiv}. This produces the directed cycle 
    \[
        \varepsilon_{0} \midnarrow \varepsilon_{1} \midnarrow \varepsilon_{2} \midnarrow \dots \midnarrow \varepsilon_{m} = \varepsilon_{0}
    \]
    which contradicts \Cref{lm:vertex_acyclic}.
\end{proof}

\begin{proof}[Proof of \Cref{prop:parallel_main}]
    Let $p$ be any vertex of $P$. It follows from \Cref{lm:type_acyclic} that there exists some type $i \in [r]$ such that 
    $\varepsilon \midarrow \delta$ for all $\varepsilon, \delta \in E(p)$
    with $\delta$ of type $i$ and $\varepsilon$ not of type $i$. We claim that this value of $i$ is suitable.

    To see this, let $\xi$ be the vertex of $\Delta_i$ such that $p \in P_{\xi}$
    and let $\zeta \neq \xi$ be any other vertex of $\Delta_i$. There is a unique edge $\varepsilon$ in $P$ one of whose endpoints is $p$ and the other of which lies in $P_{\zeta}$. Let $q$ be this other endpoint of $\varepsilon$.
    Among the edges adjacent to $p$ in $P$, those lying in $P_{\xi}$ are precisely the ones \emph{not} of type $i$. The same holds at $q$ as regards $P_{\zeta}$. Moreover, the correspondence $\delta \leftrightsquigarrow \tau_{\varepsilon} \delta$ is a bijection between the edges adjacent to $p$ in $P_{\xi}$
    and the edges adjacent to $q$ in $P_{\zeta}$.
    In each case, the assumption that $\varepsilon$ moves $\delta$ means that $\delta$ and $\tau_{\varepsilon} \delta$ are parallel.
    Since the edges in $E(p)$ not of type $i$ span $\Aff P_{\xi}$---and since the same holds at $q$ as regards $P_{\zeta}$---it follows that $P_{\zeta}$ is parallel to $P_{\xi}$.
\end{proof}

\begin{proof}[Proof of \Cref{thm:2faces}\ref{thm:combprodsimp}]
We may assume that $P$ is a full-dimensional polytope in $\R^n$.
Choose a combinatorial isomorphism $P \simeq \Delta_1\times \dots \times \Delta_r$ as above.
Let $p_0$ be any vertex of $P$ and let $i$ be as in  \Cref{prop:parallel_main}. Without loss of generality, we may assume that $i = 1$.
Let $p_1, \dots, p_m$ be those vertices of $P$ adjacent to $p_0$ along an edge of type $1$ and let $p_{m+1}, \dots, p_n$ be the remaining neighbors of $p_0$. After an affine transformation, we may assume that $p_0$ lies at the origin and that
$p_1, \dots, p_n$ is the standard basis of $\R^{n}$. Let ${\xi}_0, \dots, {\xi}_m$ be the vertices of $\Delta_1$ such that $p_j \in P_{{\xi}_j}$ for $j = 0, \dots, m$. The polytope $P_{{\xi}_0}$ lies in the subspace $L \deq \{0\}^m \times \R^{n - m}$. For each $j = 1, \dots, m$, the polytope $P_{{\xi}_j}$ is parallel to $P_{{\xi}_0}$ and contains $p_j$, so lies in the affine subspace $L + p_j$. Since all vertices of $P$ lie in one of $P_{{\xi}_0}, \dots, P_{{\xi}_m}$, it follows that $P$ is the convex hull of these, and so is equal to the Cayley sum $\Cay(P_{{\xi}_0}, \dots, P_{{\xi}_m})$, up to permuting coordinates.

We check that this Cayley sum is coherent.
By symmetry, it suffices to check that $P_{{\xi}_0}$ and $P_{{\xi}_1}$ are Minkowski equivalent. 
Both polytopes are canonically combinatorially isomorphic to 
$\Delta_{2} \times \dots \times \Delta_r$, so are canonically combinatorially isomorphic to one another. We show that this isomorphism satisfies the criterion in \Cref{fact:pos_parallel}, which will complete the proof.
Let 
$pq$ be an edge in $P_{{\xi}_0}$. Let $p'q'$ be the corresponding edge in $P_{{\xi}_1}$.
Then $p, q, q', p'$ are the vertices in cyclic order of a quadrilateral face of $P$. Since $pq$ and $p'q'$ are the intersections of the affine plane spanned by this quadrilateral with the parallel subspaces $\Aff P_{{\xi}_0}$ and $\Aff P_{{\xi}_1}$, they are themselves parallel. Since both edges lie on the same side of the edge $pp'$, they are \emph{positively} parallel.

We now have that $P$ is the coherent Cayley sum of $P_{{\xi}_0}, \dots, P_{{\xi}_m}$. By induction on dimension, we may assume that $P_{{\xi}_0}$ is affinely isomorphic to a Cayley tower. Hence, acting by an affine isomorphism in the subspace spanned by $p_{m+1}, \dots, p_n$, we can place ourselves in the situation where $P_{{\xi}_0}$ is a Cayley tower. The Minkowski equivalent polytopes $P_{{\xi}_1}, \dots, P_{{\xi}_m}$ are then also Cayley towers, up to translation. We conclude that $P$ is a Cayley tower (up to affine isomorphism). 
\end{proof}

\begin{proof}[Proof of \Cref{thm:struct_thm_smooth}]
   The proof strategy follows closely that of \Cref{thm:2faces}\ref{thm:combprodsimp}. We assume the same notation and highlight the additional subtleties.
   Again, we may assume that $P$ is full-dimensional and that a combinatorial isomorphism 
   $P \simeq \Delta_1 \times \dots \times \Delta_r$ is given. Again, we may assume that the faces 
   $P_{\xi}$ of $P$ for ${\xi} \in \Delta_1$ are all parallel.
   Again, we let $p_0$ be a vertex and list its neighbors as 
   $p_1, \dots, p_n$ in such a way that those in the initial segment
   $p_1,\dots,p_m$ are precisely its neighbors of type 1. Since $P$ is smooth at $p_0$, we may arrange by a unimodular transformation that $p_0$ be the origin and that $p_i$ lie on the positive $i$-th coordinate axis for $i = 1,\dots,n$.
   The simplex $\Delta = \Conv(p_0, \dots, p_m)$ is a face of $P$, so is smooth. By \Cref{fact:smooth_simplex}, there is a positive integer $k$ such that $p_i = k e_i$ for $1 \leq i \leq m$.
   We conclude that 
   $P = \Cay_k(P_{{\xi}_0}, \dots, P_{{\xi}_m})$, up to permuting coordinates.
   
   Just as in the previous proof, one shows that $P_{{\xi}_0}, \dots, P_{{\xi}_m}$ are Minkowski equivalent.
   By induction on dimension, we may assume that $P_{{\xi}_0}$ is lattice isomorphic to an integral Cayley tower. Hence, acting by a lattice isomorphism in the subspace spanned by $p_{m+1}, \dots, p_n$, we can place ourselves in the situation where $P_{{\xi}_0}$ is an integral Cayley tower. 
   The Minkowski equivalent lattice polytopes $P_{{\xi}_1}, \dots, P_{{\xi}_m}$ are therefore also integral Cayley towers, up to translation. We conclude that $P$ is an integral Cayley tower (up to lattice isomorphism). 
\end{proof}

\section{Weak Minkowski summands} \label{sec:WMS_of_SCPS}

\subsection{Chimney polytopes}
In what follows, we use the term \define{affine functional} for the sum $f$ of a linear functional $\lambda$ on $\R^n$ and a constant $c \in \R$.
We call $\lambda$ the \define{linear component} of $f$ and $c$ the \define{constant component} of $f$.
We say that $f$ is \define{integral} if it takes integer values on the lattice $\mathbb{Z}^n \subset \mathbb{R}^n$, which implies in particular that $c$ is an integer.

\begin{Df}\label[Df]{df:chimney}
For a polytope $Q$ in $\R^n$ and affine functionals $f, g$ on $\mathbb{R}^n$ such that $f \geq g$ holds on $Q$, we define\footnote{The choice to put the new $\R$ factor at the start conflicts with the convention in \cite{haase21}. However, it makes the statement of \Cref{lm:CCCC} much cleaner.}
\[
    \Ch{f}{g}{Q} \deq \{(t, x) \in \R \times \R^n : x \in Q, g(x) \leq t \leq f(x) \}.
\]
We call $\Ch{f}{g}{Q}$ a \define{chimney polytope} over $Q$.
We say that the chimney $\Ch{f}{g}{Q}$ is \define{strict} if the strict inequality $f > g$ holds on $Q$. 
% When we omit one of $f$ or $g$,\footnote{It will invariably be $g$ in practice.} the intended meaning is that that functional is 0.
When $Q$ is a lattice polytope and both $f$ and $g$ are integral, we call $\Ch{f}{g}{Q}$ an \define{integral chimney} over $Q$.
\end{Df}
\noindent
$\Ch{f}{g}{Q}$ may alternatively be defined as the polytope with vertices $(g(q), q)$ and $( f(q), q)$
for vertices $q$ of $Q$.
The term \emph{chimney polytope} was coined in \cite{haase21}*{\textsection~2.2.1}, 
where it is however used to mean what we have termed \emph{integral chimney} polytopes.
We will use the symbol $x_0$ for the new first coordinate in $\R \times \R^n$ and $e_0$ for the corresponding standard basis vector.
In all that follows, unless explicitly stated otherwise, any use of the notation $\Ch{f}{g}{Q}$ includes within it the assumption that $Q$ is a polytope, that $f$ and $g$ are affine functionals on its ambient space, and that the inequality $f \geq g$ holds on $Q$.

Let $P = \Ch{f+b}{g+c}{Q}$ be a strict chimney polytope where $f$ and $g$ are linear functionals and $b,c$ are constants. The polytope $P$ is combinatorially isomorphic to $[0, 1] \times Q$. An explicit combinatorial isomorphism sends the vertex $(g(q), q)$ to $(0, q)$ and 
$(f(q), q)$ to $(1, q)$ for each vertex $q$ of $Q$.
When $f > g$ are \emph{constants}, the chimney 
$\Ch{f}{g}{Q}$ is the prism $[g, f] \times Q$.
In general, the normal fan of $\Ch{f}{g}{Q}$ is obtained by ``warping'' that of the prism.

\begin{Obs}\label[Obs]{obs:chimney_fans}
    Let $Q$ be a full-dimensional polytope in $\R^n$  and $f, g: \R^n \to \R$ affine functionals satisfying $f > g$ on $Q$. 
    Write $f = \inner{u,-} + b$, $g = \inner{v,-} + c$ for 
    $u, v \in \R^n$ and $b,c \in \R$.
    Let $P = \Ch{f}{g}{Q}$.

    The normal fan $\Sigma_Q$ naturally lives in the 
    subspace $\R^n$ of $\R \times \R^n$.
    In this sense, the rays of $\Sigma_P$ are precisely the rays of $\Sigma_Q$ along with two further rays generated by vectors
    $\rho_0, \rho_1 \in \R \times \R^n$ where 
    \[
        \rho_0 = (-1, v), \qquad 
        \rho_1 = (1, -u).
    \]
    The cones of $\Sigma_P$ are given by the positive hulls of the sets of rays of the form $C\cup S$ where $C$ is the set of rays of a cone in $\Sigma_Q$ and $S$ is a proper subset of $\{\rho_0, \rho_1\}.$
    The support function $h_{P}$ is given on the rays of $\Sigma_P$ by
    \[
        h_P(\rho) = \begin{cases}
            h_Q(\rho) & \text{if $\rho$ is a ray of $\Sigma_Q$,} \\
            -c & \text{if } \rho = \rho_0,\\
            b & \text{if } \rho = \rho_1.
        \end{cases} 
    \]
    An example is depicted in \Cref{fig:chimney}.
\end{Obs}

\begin{figure}[ht]
    \begin{subfigure}[t]{0.22\textwidth}
        % \centering
        \resizebox{\linewidth}{!}{
            \begin{tikzpicture}[>=Stealth, line width=1.4pt, scale=1.6]

\definecolor{sketchpurple}{HTML}{9C27B0}
\definecolor{sketchorange}{HTML}{E8973A}
\definecolor{sketchgreen}{HTML}{4CAF50}
\definecolor{sketchblue}{HTML}{3B6FD4}

  % --- Polygon P ---
  % Left edge vertical and short, right edge vertical and tall,
  % top edge slanting up to the right, bottom edge slanting slightly down.
  \coordinate (A) at (0,0);       % bottom-left
  \coordinate (B) at (0,1.1);     % top-left
  \coordinate (C) at (1.7,2.2);   % top-right
  \coordinate (D) at (1.7,-0.35); % bottom-right
  \draw[sketchpurple] (A) -- (B) -- (C) -- (D) -- cycle;
  \node at (-0.55,0.55) {\Large $P$};

  % --- Outward normal u of the top edge ---
  % Top edge direction (1.7,1.1); outward normal ~ (-1.1,1.7), angle ~122.9 deg.
  \coordinate (mtop) at ($(B)!0.5!(C)$);
  \draw[sketchorange,->] ($(mtop)+(122.9:0.12)$) -- ($(mtop)+(122.9:0.75)$);
  \node[sketchorange] at ($(mtop)+(122.9:1.0)+(-0.15,0)$) {\Large $\rho_1$};

  % --- Outward normal v of the bottom edge ---
  % Bottom edge direction (1.7,-0.35); outward normal ~ (-0.35,-1.7), angle ~-101.6 deg.
  \coordinate (mbot) at ($(A)!0.5!(D)$);
  \draw[sketchgreen,->] ($(mbot)+(-101.6:0.12)$) -- ($(mbot)+(-101.6:0.85)$);
  \node[sketchgreen] at ($(mbot)+(-101.6:1.15)$) {\Large $\rho_0$};

  % --- Segment Q ---
  \draw[sketchblue] (0,-2) -- (1.5,-2);
  \node at (-0.55,-2) {\Large $Q$};
\end{tikzpicture}
        }
    \end{subfigure}
    \hspace{0.1\textwidth}
    \begin{subfigure}[t]{0.22\textwidth}
        % \centering
        \resizebox{\linewidth}{!}{
            \definecolor{sketchcyan}{HTML}{4DC9E6}
\definecolor{sketchorange}{HTML}{E8973A}
\definecolor{sketchgreen}{HTML}{4CAF50}

\begin{tikzpicture}[>=Stealth, line width=1.4pt]
  % --- N(P) ---
  \node at (-2.8,0) {\Large $\Sigma_P$};
  \draw[sketchcyan,<->] (-1.1,0) -- (1.1,0);
  \draw[sketchorange,->] (0,0) -- (122.9:1.0);
  \node[sketchorange] at ($(122.9:1.3)+(-0.15,0)$) {\Large $\rho_1$};
  \draw[sketchgreen,->] (0,0) -- (-101.6:1.1);
  \node[sketchgreen] at (-101.6:1.4) {\Large $\rho_0$};
  \fill (0,0) circle (1.8pt);

  % --- N(Q) ---
  \node at (-2.8,-6) {\Large $\Sigma_Q$};
  \draw[sketchcyan,<->] (-1.1,-6) -- (1.1,-6);
  \fill (0,-6) circle (1.8pt);
\end{tikzpicture}
        }
    \end{subfigure}
    \caption{An example of a chimney $P = \Chim_g^f(Q) \sbs \R \times \R$ with base polytope $Q \sbs \R$, and their normal fans.}
    \label{fig:chimney}
\end{figure}
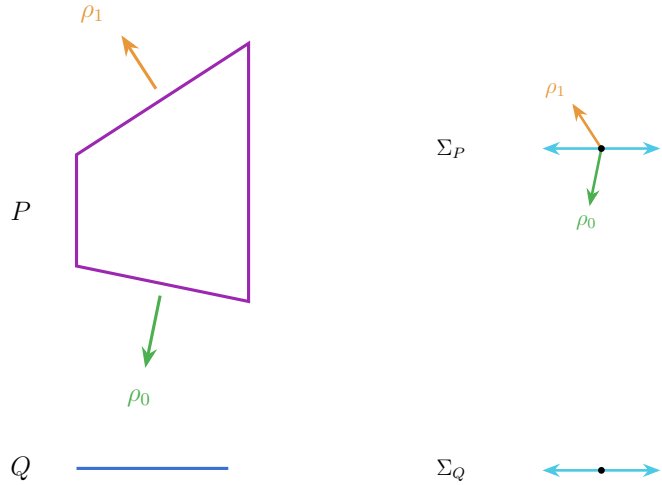

\begin{Lm}\label[Lm]{lm:MEQchimney}
    Let $P$ be a polytope in $\mathbb{R}^n$.
    The polytopes in $\R \times \mathbb{R}^{n}$ Minkowski equivalent to a given strict chimney polytope $\Ch{f}{g}{P}$ over $P$ are precisely the strict chimney polytopes $\Ch{f+b}{g+c}{Q}$ in which $Q$ is Minkowski equivalent to $P$ and $b,c \in \mathbb{R}$ are chosen so that $f+b > g + c$ on $Q$.
\end{Lm}
\begin{proof}
    The fact that such polytopes are Minkowski equivalent to $\Ch{f}{g}{P}$ follows from the fact that the description of the normal fan in \Cref{obs:chimney_fans} depends only on the normal fan of the base polytope and the linear components of the two affine functionals.
    
    For the converse, suppose we are given a polytope $\widetilde{Q}$ Minkowski equivalent to $\Ch{f}{g}{P}$.
    The projection of $\widetilde{Q}$ to the second factor of $\R \times \R^n$ is a polytope $Q$ in $\R^n$ Minkowski equivalent to $P$ by \Cref{prop:MEQaffine}. 
    From the description of (the rays of) the normal fan in \Cref{obs:chimney_fans}, one sees that $\widetilde{Q}$ is the intersection with 
    $\mathbb{R} \times Q$ of the halfspaces 
    $\{x_0 \leq f(x_1, \dots, x_n) + b\}$ and $\{x_0 \geq g(x_1, \dots, x_n) + c\}$ for some $b, c \in \mathbb{R}$.
    To have the projection to $\R^n$ be $Q$, we must have 
    $f + b \geq g + c$ on $Q$. 
    If equality were attained at any point in $Q$, then the upper and lower facets of $\widetilde{Q}$ would intersect. This would be detected in the normal fan by the corresponding rays lying in a common cone. This, however, is false. Hence $\widetilde{Q} = \Ch{f+b}{g+c}{Q}$ is a strict chimney.
\end{proof}
\noindent
We can similarly characterize weak Minkowski summands of strict chimneys.
\begin{Lm}\label[Lm]{lm:chimney_summand}
    Let $P$ be a polytope in $\R^n$.
    The weak Minkowski summands in $\R\times \mathbb{R}^{n}$
    of a given \emph{strict} chimney polytope $\widetilde{P} \deq \Ch{f}{g}{P}$ over $P$ are precisely the polytopes $\widetilde{Q} = \Ch{f+b}{g+c}{Q}$ where $Q$ is a weak Minkowski summand of $P$, and 
    $b, c \in \mathbb{R}$ are chosen so that $f + b \geq g + c$ holds on $Q$.
\end{Lm}
\begin{proof}
    After translating $P$, we may assume that it contains the origin.
    To see that each polytope described in the lemma statement is indeed a weak Minkowski summand of $\widetilde{P}$, let $\epsilon_1, \epsilon_2, \dots$ be a decreasing sequence of positive reals with limit $0$.
    Choose $\delta_i$ tending to $0$ and sufficiently small so that $f+b+\epsilon_i > g + c$ holds on $Q + \delta_i P$.
    Then
    $\widetilde{Q} = \bigcap_{i} \Ch{f+b + \epsilon_i}{g+c}{Q + \delta_i P}$. Each polytope in the intersection is Minkowski equivalent to $\widetilde{P}$ by \Cref{lm:MEQchimney} and \Cref{prop:nef}. Hence $\widetilde{Q}$ is a weak Minkowski summand of $\widetilde{P}$ by \Cref{lm:Minkowski_limit}.    
    
    It remains to show that all weak Minkowski summands have this form.
    Suppose we are given a weak Minkowski summand $\widetilde{Q}$ of $\Ch{f}{g}{P}$.
    The projection of $\widetilde{Q}$ to the second factor of $\R \times \R^n$ is a weak Minkowski summand $Q$ of $P$ by \Cref{prop:WMSaffine}. The description of (the rays of) the normal fan in \Cref{obs:chimney_fans} shows that $\widetilde{Q}$ is the intersection with 
    $\mathbb{R} \times Q$ of the halfspaces 
    $\{x_0 \leq f(x_1, \dots, x_n) + b\}$ and $\{x_0 \geq g(x_1, \dots, x_n) + c\}$ for some $b, c \in \mathbb{R}$.
    To have the projection to $\R^n$ be $Q$, we must have 
    $f + b \geq g + c$ on $Q$. Hence $\widetilde{Q} = \Ch{f+b}{g+c}{Q}$.
\end{proof}

\begin{Cor}\label[Cor]{cor:integral_chimney_summand}
    Let $P$ be a lattice polytope in $\R^n$ and $f,g: \R^n \to \R$ \emph{integral} affine functionals satisfying $f > g$ on $P$.
    The lattice weak Minkowski summands of the strict chimney $\Ch{f}{g}{P}$ are precisely the polytopes $\Ch{f+b}{g+c}{Q}$ where $Q$ is a lattice weak Minkowski summand of $P$ and 
    $b,c \in \mathbb{Z}$ are such that $f+b\geq g+c$ on $Q$.
\end{Cor}
\begin{proof}
    Given \Cref{lm:chimney_summand}, it remains only to check that
    $\widetilde{Q} \deq \Ch{f+b}{g+c}{Q}$ is a lattice polytope if and only if $Q$
    is lattice and both $b$ and $c$ are integers. Sufficiency is clear. We prove necessity.
    Let $q$ be a vertex of $Q$. Since $(f(q) + b, q)$ is a vertex of $\widetilde{Q}$ and $f$ is integral, $q$ must be a lattice point and $b$ an integer. For similar reasons, $c$ must also be an integer.
\end{proof}
\noindent
In contrast to \Cref{obs:integral=lattice}, a chimney polytope which is a lattice polytope need not be an integral chimney. Instead, we will use the following technical criterion to ensure that our chimneys are integral (when they need to be). 
\begin{Lm}\label[Lm]{lm:smooth_chimney}
Let $P$ be a smooth lattice polytope in $\mathbb{R}^{n+1}$ and suppose that $P$ is affinely isomorphic to 
$\Ch{f}{g}{Q}$ for some polytope $Q$ in $\mathbb{R}^n$ and affine functionals $f, g$ (with no integrality assumptions on these).
Then $P$ is lattice isomorphic to $\Ch{f'}{0}{Q'}$ for some smooth lattice polytope $Q'$ affinely isomorphic to $Q$ in $\mathbb{R}^n$ and $f'$ an integral affine functional.
\end{Lm} 
\begin{proof}
    There is no loss of generality in assuming that $P$, and therefore also $Q$, is full-dimensional.
    Let $P^{\up}$ and $P^{\down}$ be the faces of $P$ that correspond under the affine isomorphism to the top and bottom facets of $\Ch{f}{g}{Q}$, respectively.
    Acting by a lattice-preserving affine automorphism of $\mathbb{R}^{n+1}$, we may assume that $P^{\down}$ lies in $\{0\} \times \mathbb{R}^n$, which we identify with $\mathbb{R}^n$. Let $q$ be a vertex of $Q$ such that the copies $q^{\up}$ and $q^{\down}$ of $q$ in $P^{\up}$ and $P^{\down}$ are distinct. Let $v$ be the primitive edge direction from $q^{\down}$ to $q^{\up}$. Since $P$ is smooth, we can find a lattice-preserving affine automorphism fixing  $P^{\down}$ and sending $v$ to the basis element $e_{0}$; we may therefore assume $v = e_{0}$.

    The hypothesis on $P$ now implies that 
    $P = \Ch{f_0}{0}{P^{\down}}$ for some affine functional $f_0$. 
    Since $P$ is smooth, the lattice 
    $\mathbb{Z}^{n+1} \cap \Aff P^{\up}$ is a complement in $\mathbb{Z}^{n+1}$
    to the sublattice $\mathbb{Z} \times \{0\}^{n}$.
    Hence the projection map $\mathbb{R}^{n+1} \to \mathbb{R}^n$ maps the lattice points of $P^{\up}$ bijectively onto the lattice points of $P^{\down}$.
    Since $P^{\down}$ is smooth 
    and $P^{\up}$ is the image of $P^{\down}$ under the linear map 
    $x \mapsto (f_0(x), x)$,  $f_0$ must be integral on the lattice $\mathbb{Z}^{n} \cap \Aff P^{\down}$. We can therefore find an integral functional $f'$ on $\mathbb{R}^n$ which agrees with 
    $f_0$ on $P^{\down}$. 
    Then $P = \Ch{f'}{0}{P^{\down}}$.
    Since $P^{\down}$ is a face of the smooth polytope $P$, it is itself smooth. So we may take $Q' = P^{\down}$.
\end{proof}
\noindent
A partial converse to \Cref{lm:smooth_chimney} is given in the following lemma.
\begin{Lm}\label[Lm]{lm:strict=smooth}
    A strict integral chimney over a smooth polytope is smooth.
\end{Lm}
\begin{proof}
    Let $Q = \Ch{f}{g}{P}$ be a strict integral chimney over a smooth polytope $P$ in $\mathbb{R}^n$. We show that $Q$ is smooth at each of its vertices. By symmetry, it suffices to do so at vertices of the form $(g(p), p) \in \mathbb{R} \times \mathbb{R}^n$ for vertices $p$ of $P$. Applying the lattice isomorphism $(t, x) \mapsto (t - g(x), x)$, we may assume that $g = 0$. The edges in $Q$ incident to the vertex $(0, p)$  are those of the smooth polytope $\{0\} \times P$ along with the additional edge to $(f(p), p)$. The primitive edge directions for the first group form a basis for the sublattice $\{0\} \times \mathbb{Z}^n$. The primitive vector for the last edge is $(1,0)$. Together, they form a basis for the lattice $\mathbb{Z}^{n+1}$.
\end{proof}

\begin{Df}\label[Df]{df:iterated_chimney}
A polytope $P$ is said to be an \define{iterated chimney polytope}
if there exists a sequence of polytopes 
$P_0, P_1, \dots, P_m = P$ such that 
$P_0$ is a single point and $P_{i+1}$ is a chimney polytope over $P_i$ for each $i$.
If $P$ is a lattice polytope and if we may in each case choose $P_{i+1}$ to be an \emph{integral} chimney polytope over $P_{i}$, then we say that $P$ is \define{Nakajima}.\footnote{In lieu of the more cumbersome ``iterated integral chimney polytope''.}
\end{Df}
\noindent
As mentioned in the introduction, the nomenclature reflects the origin of these polytopes in the work of H.\ Nakajima \cite{Nakajima} classifying complete intersection affine toric varieties. Iterated chimneys which are lattice polytopes need not be Nakajima. However,
an easy induction based on \Cref{lm:smooth_chimney} proves:
\begin{Cor}\label[Cor]{cor:smooth=>Nakajima}
    Any smooth lattice polytope affinely isomorphic to an iterated chimney is lattice isomorphic to a Nakajima polytope.
\end{Cor}
\begin{Prop}\label[Prop]{prop:WMS_of_Nakajima}
    Every weak Minkowski summand of an iterated chimney polytope is again an iterated chimney polytope.
    Every lattice weak Minkowski summand of a Nakajima polytope is again a Nakajima polytope.
\end{Prop}
\begin{proof}
We proceed by induction on dimension. 
Let $P$ be an iterated chimney polytope.
We may write $P = \Ch{f}{g}{P_0}$ for some iterated chimney polytope $P_0$ and affine functionals $f,g$.
Then $P' \deq \Ch{f+1}{g}{P_0}$ is a strict chimney polytope over $P_0$. The polytope $P$ is a weak Minkowski summand of $P'$ by \Cref{lm:chimney_summand}. Since any weak Minkowski summand of $P$ is a weak Minkowski summand of $P'$, it suffices to prove the claim for $P'$. By another use of \Cref{lm:chimney_summand}, if $Q$ is a weak Minkowski summand of $P'$ then  
$Q = \Ch{f+b}{g+c}{Q_0}$ for some weak Minkowski summand $Q_0$ of $P_0$ and constants $b,c \in \mathbb{R}$.
By induction, $Q_0$ must be an iterated chimney polytope. Hence so is $Q$.

For the second claim, the proof proceeds along similar lines. In this case, the polytope $P_0$ is assumed Nakajima and $f,g$ are assumed integral. Then $P'$ will be Nakajima as well.
By \Cref{cor:integral_chimney_summand}, 
a lattice weak Minkowski summand $Q$ of $P'$ takes the form 
$\Ch{f+b}{g+c}{Q_0}$ for some \emph{integers} $b, c$ and $Q_0$ a lattice weak Minkowski summand of $P_0$.
By the induction hypothesis, $Q_0$ is Nakajima. Hence $Q$ is Nakajima.
\end{proof}
\noindent
The following lemma shows that the chimney and Cayley operations commute.
\begin{Lm}\label[Lm]{lm:CCCC}
    Let $P, Q$ be polytopes in $\R^{n}$, $f, g$ affine functionals on $\R^{n}$,
    $b,c\in \mathbb{R}$ arbitrary and $k$ a positive real.
    Assume $f \geq g$ holds on $P$ and 
    $f+b \geq g+c$ holds on $Q$.
    Then
    \[
        \Cay_k\left(\Ch{f}{g}{P}, \Ch{f+b}{g+c}{Q}\right)
        = \Ch{\tilde f}{\tilde g}{\Cay_k(P, Q)}
    \]
    where $\tilde f, \tilde g$ are the affine-linear functionals on $\R^{n} \times \R$ given by
    \[
        \tilde f: (x, t) \mapsto f(x) + \frac{b}{k}t, 
        \qquad 
        \tilde g: (x, t) \mapsto g(x) + \frac{c}{k}t.
    \]
    In particular, $\tilde f \geq \tilde g$ holds on $\Cay_k(P,Q)$.
\end{Lm}
\begin{proof}
    The last statement may be readily verified on the vertices of $\Cay_k(P,Q)$, which are those of $P \times \{0\}$ and $Q \times \{k\}$.
    Similarly, since the polytopes being equated both have their vertices lying in the union of the two affine hyperplanes $\{x_{n+1} = 0\}$ and 
    $\{x_{n+1} = k\}$, it suffices to check that their intersections with each of these hyperplanes are equal.
    The intersections with $\{x_{n+1} = 0\}$ are 
    $\Ch{f}{g}{P}$ and $\Ch{\tilde f}{\tilde g}{P}$.
    Those with $\{x_{n+1} = k\}$ are
    $\Ch{f+b}{g+c}{Q}$ and $\Ch{\tilde f}{\tilde g}{Q}$.
    Since $\tilde f$ and $\tilde g$ restrict on $\{x_{n+1} = 0\}$ to $f$ and $g$ respectively, and on
    $\{x_{n+1} = k\}$ to $f + b$ and $g + c$ respectively, they are in agreement.
\end{proof}
\noindent
Using \Cref{obs:nested_cayley}, \Cref{lm:CCCC} easily generalizes to arbitrary generalized  Cayley sums.
\begin{Cor}\label[Cor]{cor:gCCCC}
    Let $P_0, \dots, P_m$ be polytopes in $\R^{n}$, $f, g$ affine functionals on $\R^{n}$, $k$ a positive real and 
    $b_0, \dots, b_m, c_0, \dots, c_m \in \mathbb{R}$ arbitrary.
    Then there exist affine functionals $\tilde f, \tilde g$ on $\mathbb{R}^n \times \mathbb{R}^m$ which depend only on $f,g$, $k$ and the constants $b_i, c_j$ such that 
    \[
        \Cay_k\left( \Ch{f + b_0}{g + c_0}{P_0}, \dots, \Ch{f+b_m}{g+c_m}{P_m} \right)
        = \Ch{\tilde f}{\tilde g}{\Cay_k(P_0, \dots, P_m)}.
    \]
\end{Cor}

\begin{Prop}\label[Prop]{prop:Cayley_commute}
    A coherent Cayley sum of iterated chimney polytopes is again an iterated chimney polytope.
\end{Prop}
\begin{proof}
    We proceed by induction on the dimension of the Cayley sum $P$.
     We may assume $\dim P > 0$ and write $P = \Cay_k(Q_0, \dots, Q_m)$ for a positive real number $k$ and Minkowski equivalent iterated chimney polytopes $Q_0, \dots, Q_m$.
    We have $Q_0 = \Ch{f}{g}{R}$ for some iterated chimney polytope $R$ and affine functionals $f,g$.
    For each $i$, $Q_i$ is Minkowski equivalent to $Q_0$ which is a weak Minkowski summand of $\Ch{f+1}{g}{R}$.
    We can therefore apply \Cref{lm:chimney_summand} to write 
    $Q_i = \Ch{f + b_i}{g + c_i}{R_i}$ for some real numbers $b_1, \dots, b_m, c_1, \dots, c_m$ and weak Minkowski summands $R_1, \dots, R_m$ of $R$. 
    By considering the projections of $Q_i$ down to $R_i$ and using \Cref{prop:MEQaffine}, we conclude that $R_i$ is in fact Minkowski equivalent to $R$.
    By \Cref{prop:WMS_of_Nakajima}, each $R_i$ is an iterated chimney polytope. By \Cref{cor:gCCCC}, we have
    \[
        P = \Cay_k\left( \Ch{f}{g}{R}, \Ch{f+b_1}{g+c_1}{R_1}, \dots, \Ch{f+b_m}{g+c_m}{R_m} \right)
        = \Ch{\tilde f}{\tilde g}{\Cay_k(R, R_1, \dots, R_m)}
    \]
    for some affine functionals $\tilde f, \tilde g$. By the induction hypothesis, $\Cay_k(R, R_1, \dots, R_m)$ is an iterated chimney polytope. Hence so is $P$.
\end{proof}
\noindent
An easy induction based on \Cref{prop:Cayley_commute} gives:
\begin{Cor}\label[Cor]{cor:tower=iterated}
    Every Cayley tower is an iterated chimney polytope.
\end{Cor}
\noindent
We are now ready to prove our main result.

\begin{proof}[Proof of \Cref{thm:Nakajima}]
    We prove the two implications 
    \ref{it:WMS_SCPS}$\Rightarrow$\ref{it:Nakajima} and 
    \ref{it:Nakajima}$\Rightarrow$\ref{it:WMS_SCC}. The remaining implication 
    \ref{it:WMS_SCC}$\Rightarrow$\ref{it:WMS_SCPS} is trivial.

    To prove \ref{it:WMS_SCPS}$\Rightarrow$\ref{it:Nakajima}, it will suffice by \Cref{prop:WMS_of_Nakajima} to show that a smooth combinatorial product of simplices is lattice isomorphic to a Nakajima polytope.
    We argue that all smooth integral Cayley towers are Nakajima up to a lattice isomorphism, which suffices by \Cref{thm:struct_thm_smooth}.
    This follows from the combination of \Cref{cor:tower=iterated} and \Cref{cor:smooth=>Nakajima}.
    
    To prove \ref{it:Nakajima}$\Rightarrow$\ref{it:WMS_SCC}, we need only show that every Nakajima polytope is a weak Minkowski summand of a smooth combinatorial cube. 
    Now for any Nakajima polytope $P$ in $\mathbb{R}^n$, we can find a sequence of Nakajima polytopes 
    $P_0, P_1, \dots, P_n = P$ where 
    $P_0$ is a point and for each $i$ we have $P_{i+1} = \Ch{f_i}{g_i}{P_i}$ for some integral affine functionals 
    $f_i, g_i$ on $\mathbb{R}^i$.
    Now we inductively construct a sequence of polytopes 
    $P_0', \dots, P_n'$ by setting $P_0' = P_0$
    and $P_{i+1}' = \Ch{f_i + c_i}{g_i}{P_i'}$ where $c_i$ is a positive integer chosen so that $P_{i+1}'$ is a strict chimney over $P_i'$.
    Then by induction we observe:
    \begin{itemize}
        \item For each $i$, the polytope $P_i'$ is combinatorially isomorphic to an $i$-dimensional cube.
        \item For each $i$, the polytope $P_i'$ is smooth (using \Cref{lm:strict=smooth}).
        \item For each $i$, the polytope $P_i$ is a weak Minkowski summand of $P_i'$ (using \Cref{lm:chimney_summand}).
    \end{itemize}
    Setting $i = n$ gives the desired conclusion.
\end{proof}
\noindent
The argument for \ref{it:Nakajima}$\Rightarrow$\ref{it:WMS_SCC} above  shows that every iterated chimney polytope is a weak Minkowski summand of an iterated chimney polytope combinatorially isomorphic to a cube.

\begin{proof}[Proof of \Cref{cor:unimodulartriangulation}]
By \Cref{thm:Nakajima}, $P$ is Nakajima.
Corollary 2.9 in \cite{haase21} states that every Nakajima polytope has a quadratic triangulation.
\end{proof}

\begin{Lm}\label[Lm]{lm:cayley_nakajima}
    Let $P$ be a Nakajima polytope. 
    Let $Q_1, Q_2$ be lattice weak Minkowski summands of $P$.
    Then $\Cay(Q_1, Q_2)$ is Nakajima.
\end{Lm}
\begin{proof}
    We proceed by induction on the dimension of $P$.
    The case $\dim P = 0$ is trivial. In general, write $P = \Ch{f}{g}{P_0}$ for $P_0$ a Nakajima polytope and $f, g$ integral affine functionals.
    Replacing $P$ with $\Ch{f+1}{g}{P_0}$, we may assume that $P$ is a strict chimney over $P_0$.
    By \Cref{cor:integral_chimney_summand}, we have 
    $Q_i = \Ch{f+b_i}{g+c_i}{R_i}$ for $i = 1,2$
    where $b_i$ and $c_i$ are integers and $R_i$ is a lattice weak Minkowski summand of $P_0$.
    By \Cref{lm:CCCC}, we have
    \[
        \Cay(Q_1, Q_2) = \Cay\left(\Ch{f + b_1}{g + c_1}{R_1}, \Ch{f+b_2}{g+c_2}{R_2}\right)
        = \Ch{\tilde f}{\tilde g}{\Cay(R_1, R_2)}
    \]
    where $\tilde f, \tilde g: \mathbb{R}^n \times \mathbb{R} \to \mathbb{R}$ are given by
    \[
        \tilde f(x, t) = f(x) + b_1 + (b_2 - b_1)t, 
        \qquad 
        \tilde g(x, t) = g(x) + c_1 + (c_2 - c_1)t.
    \]
    Observing that $\tilde f, \tilde g$ are integral and that $\Cay(R_1, R_2)$ is Nakajima by the induction hypothesis, we conclude that $\Cay(Q_1, Q_2)$ is Nakajima.
\end{proof}

\begin{proof}[{Proof of \Cref{thm:absolute_ODA}}]
    It suffices to prove the first claim since the second follows by taking $Q' = kQ$ for positive integers $k$.
    We employ the so-called \emph{Cayley trick} to reduce from two polytopes to one. 
    Namely, two polytopes are relatively IDP provided their Cayley sum is IDP (see,  e.g., \cite{T23}*{Theorem 1.3(1)}). 
    By \Cref{thm:Nakajima}, we may assume that $Q$ and $Q'$ are Nakajima.
    By \Cref{lm:cayley_nakajima}, 
    $\Cay(Q, Q')$ is then also Nakajima. By \Cref{cor:unimodulartriangulation}, it therefore has a unimodular triangulation.
    Since a unimodular simplex is IDP, and since a lattice polytope covered by IDP lattice polytopes is IDP, $\Cay(Q, Q')$ is IDP.
\end{proof}

\section{Relationship to the literature} \label{sec:literature}

In this section, we interpret the preceding theorems and constructions in the language of toric varieties, thereby linking our results to others in the literature.
Indeed, the disparate presentations of results on this topic in the literature have proven a surprising hindrance to a full understanding of what is known about smooth combinatorial products of simplices and their weak Minkowski summands. 
Our aim in this section is, in part, to begin to remedy this state of affairs.
A full treatment of toric geometry and its connection to the combinatorial theory of polytopes is beyond the scope of this short section; we recommend the article \cite{Cox03} for an approachable introduction, and the books \cites{fulton, CLS11} for in-depth expositions.

For simplicity, we work in the category of complex varieties.
As explained in \cite{fulton}*{Section 1}, to each strictly convex rational fan $\Sigma$ in $\mathbb{R}^n$ one associates a normal toric variety $X_{\Sigma}$. Conversely, every normal toric variety arises from such a fan, which is unique up to lattice isomorphism.
The toric variety $X_{\Sigma}$ is \define{affine} if and only if $\Sigma$ consists of a single cone and its faces, and it is \define{complete} if and only if $\Sigma$ is complete. 
$X_{\Sigma}$ is \define{projective} if and only if $\Sigma = \Sigma_P$ for some full-dimensional polytope $P$ in $\R^n$; 
in this case, the polytope can always be chosen lattice.
The toric variety $X_{\Sigma}$ is \define{smooth} if and only if $\Sigma$ is smooth. 

A lattice polytope $P$ in $\mathbb{R}^n$ determines a \define{polarization}, i.e.\ an ample toric line bundle $\mathcal{L}_P$ on $X_P \deq X_{\Sigma_P}$.
If $P$ is smooth, $\mathcal{L}_P$ will in fact be very ample.
In general, this will be true after rescaling $P$ by a large enough positive integer $k$; this corresponds to replacing $\mathcal{L}_P$ with the tensor power 
$\mathcal{L}_P^{\otimes k}$. If $\mathcal{L}_P$ is very ample, the corresponding complete linear system equivariantly embeds $X_P$ inside a projective space $\mathbb{P}^N$. The affine cone $CX_P$ over $X_P$ in $\mathbb{A}^{N+1}$ under this embedding is again a toric variety, but may not be normal.
In fact, normality of $CX_P$ is \emph{equivalent} to $P$ being IDP.
The \define{normalization} $\widetilde{CX_P}$ is an affine normal toric variety. It corresponds to the cone in $\mathbb{R}^{n+1}$ which is the \define{epigraph} of the support function $h_{P}$ on $\Sigma_P \subseteq \mathbb{R}^n$.
Nakajima \cite{Nakajima} showed that $\widetilde{CX_P}$ is a complete intersection variety if and only if $P$ is Nakajima.

Suppose $X = X_P$ is a projective normal toric variety given by a complete rational fan $\Sigma$. Let $P_0, \dots, P_m$ be lattice polytopes with normal fan $\Sigma$.
Each $P_i$ gives a toric line bundle $\mathcal{L}_{i} \deq \mathcal{L}_{P_i}$.
Let $\mathcal{E}$ be the total space of the vector bundle $\mathcal{L}_0 \oplus \dots \oplus \mathcal{L}_m$. The relative projectivization 
$\mathbb{P}_X\mathcal{E}$ (which projectivizes each fiber of $\mathcal{E} \to X$) is then again a toric variety with a canonical polarization given by its tautological line bundle. The lattice polytope which represents this line bundle is the Cayley sum 
$\Cay(P_0, \dots, P_m)$; cf.\ \cite{Cox97}*{\textsection~3}. Thus the toric variety corresponding to an integral Cayley tower is one obtained by recursively taking projective bundles (coming from split toric bundles), starting with a point. The varieties constructed in this way are exactly the toric generalized Bott towers. These were studied in the more general setting of \define{quasitoric manifolds} in \cites{quasitoric1, quasitoric2}.
In this language, \Cref{thm:struct_thm_smooth} reads:
\theoremstyle{plain}
\newtheorem*{structure_thm_smooth_restatement}{Theorem \ref{thm:struct_thm_smooth}} 
\begin{structure_thm_smooth_restatement}
    Let $X$ be a smooth projective toric variety associated to a polytope $P$ which is a combinatorial product of simplices. Then $X$ is a generalized Bott tower.
\end{structure_thm_smooth_restatement}
\noindent
This is the special case for toric varieties of \cite{quasitoric2}*{Theorem 6.4}.\footnote{Strictly speaking, Choi--Masuda--Suh state that $X$ is a generalized Bott tower up to equivariant homeomorphism. From their methods however, one can deduce in this case that it is up to isomorphism of toric varieties in our setting. This is also naturally a consequence of \Cref{thm:struct_thm_smooth}.}
Without using this nomenclature, Batyrev \cite{batyrev91} had earlier proved an equivalent result. \cite{batyrev91}*{Corollary 4.4} shows that a smooth complete toric variety is a generalized Bott tower as soon as the minimal non-cones in its fan (there called \define{primitive collections}) are pairwise disjoint, the defining property of what he calls a \define{splitting fan}.
For a complete fan in $\mathbb{R}^n$, this condition says precisely that the combinatorial type of the fan is that of the normal fan of a product of simplices.

We turn now to Oda's conjectures. In the language of toric varieties, the conjectures are:
\theoremstyle{plain}
\newtheorem*{Oda_Restatement}{Conjecture \ref{conj:oda}} 
\begin{Oda_Restatement}
Let $X_P$ be a smooth projective toric variety coming from a smooth polytope $P$
corresponding to the ample toric line bundle $\mathcal{L}_P$.
\begin{enumerate}
    \item[(a)] The projective embedding of $X_P$ given by the complete linear system of $\mathcal{L}_P$ is projectively normal. Equivalently, for every positive integer $k$, the multiplication map
        $$
            H^0 ( X_P, \mathcal{L}_P) \otimes \cdots \otimes H^0 ( X_P, \mathcal{L}_P) \to 
            H^0 ( X_P, \mathcal{L}_P^{\otimes k})
        $$
    is surjective.

    \item[(b)] If $\c L'$ is any other ample toric line bundle on $X_P$, the multiplication map
        $$
            H^0 ( X_P, \mathcal{L}_P) \otimes H^0 ( X_P, \c L') \to 
            H^0 ( X_P, \c L_P \otimes \c L')
        $$
    is surjective.
    \item[(c)] If $\c N$ is a nef toric line bundle on $X_P$, the multiplication map 
        $$
            H^0 ( X_P, \mathcal{L}_P) \otimes H^0 ( X_P, \c N) \to 
            H^0 ( X_P, \c L_P \otimes \c N)
        $$
    is surjective.
\end{enumerate}
\end{Oda_Restatement} 
\noindent In this language, the translation of \Cref{thm:absolute_ODA} is:

\theoremstyle{plain}
\newtheorem*{abs_oda_restatement}{Theorem \ref{thm:absolute_ODA}} 
\begin{abs_oda_restatement}
If $\c N$ and $\c N'$ are nef toric line bundles on $X_P$, then the multiplication map 
        $$
            H^0 ( X_P, \mathcal{N}) \otimes H^0 ( X_P, \c N') \to 
            H^0 ( X_P, \c N \otimes \c N')
        $$
    is surjective.
\end{abs_oda_restatement}
\noindent This is Property 2.1 in \cite{ikeda09}.
Its validity for generalized Bott towers is precisely \cite{ikeda09}*{Corollary 4.2}.

\section*{Acknowledgments}
The authors would like to thank Teddy Gonzales for invaluable discussions.
We extend our gratitude to Sharon Robins and Mathieu Vall\'ee for pointing us to a number of useful references in the literature.
We are grateful to Matt Larson and June Huh for interesting conversations, particularly concerning \Cref{thm:struct_thm_smooth}.
Last but not least, we wish to thank Gaku Liu for his guidance throughout this project.

\bibstyle{amsalpha}
\bibliography{bib}

\begin{bibdiv}
\begin{biblist}

\bib{batyrev91}{article}{
      author={Batyrev, V.~V.},
       title={On the classification of smooth projective toric varieties},
        date={1991},
     journal={Tohoku Math. J. (2)},
      volume={43},
      number={4},
       pages={569\ndash 585},
}

\bib{Bogart_2015}{article}{
      author={Bogart, T.},
      author={Haase, C.},
      author={Hering, M.},
      author={Lorenz, B.},
      author={Nill, B.},
      author={Paffenholz, A.},
      author={Rote, G.},
      author={Santos, F.},
      author={Schenck, H.},
       title={Finitely many smooth d-polytopes with n lattice points},
        date={2015},
     journal={Israel J. Math.},
      volume={207},
      number={1},
       pages={301\ndash 329},
}

\bib{beck19}{article}{
      author={Beck, M.},
      author={Haase, C.},
      author={Higashitani, A.},
      author={Hofscheier, J.},
      author={Jochemko, K.},
      author={Katth\"{a}n, L.},
      author={Micha{\l}ek, M.},
       title={Smooth centrally symmetric polytopes in dimension 3 are {IDP}},
        date={2019},
     journal={Ann. Comb.},
      volume={23},
      number={2},
       pages={255\ndash 262},
}

\bib{Cox97}{article}{
      author={Cattani, E.},
      author={Cox, D.},
      author={Dickenstein, A.},
       title={Residues in toric varieties},
        date={1997},
     journal={Compos. Math.},
      volume={108},
       pages={35\ndash 76},
}

\bib{CLS11}{book}{
      author={Cox, D.~A.},
      author={Little, J.~B.},
      author={Schenck, H.~K.},
       title={Toric varieties},
      series={Grad. Stud. Math.},
   publisher={Amer. Math. Soc., Providence, RI},
        date={2011},
      volume={124},
}

\bib{quasitoric2}{article}{
      author={Choi, S.},
      author={Masuda, M.},
      author={Suh, D.~Y.},
       title={Quasitoric manifolds over a product of simplices},
        date={2010},
     journal={Osaka J. Math.},
      volume={47},
      number={1},
       pages={109\ndash 129},
}

\bib{quasitoric1}{article}{
      author={Choi, S.},
      author={Masuda, M.},
      author={Suh, D.~Y.},
       title={Topological classification of generalized {B}ott towers},
        date={2010},
     journal={Trans. Amer. Math. Soc.},
      volume={362},
      number={2},
       pages={1097\ndash 1112},
}

\bib{Cox03}{incollection}{
      author={Cox, D.},
       title={What is a toric variety?},
        date={2003},
   booktitle={Topics in algebraic geometry and geometric modeling},
      series={Contemp. Math.},
      volume={334},
   publisher={Amer. Math. Soc., Providence, RI},
       pages={203\ndash 223},
}

\bib{coxeter}{article}{
      author={Coxeter, H. S.~M.},
       title={Discrete groups generated by reflections},
        date={1934},
     journal={Annals of Mathematics},
      volume={35},
      number={3},
       pages={588\ndash 621},
}

\bib{curtis_cubes}{misc}{
      author={Curtis, J.},
       title={Smooth combinatorial cubes are {IDP}},
        date={2025},
        note={arXiv:2509.02960},
}

\bib{dickenstein}{article}{
      author={Dickenstein, A.},
      author={Di~Rocco, S.},
      author={Piene, R.},
       title={Classifying smooth lattice polytopes via toric fibrations},
        date={2009},
        ISSN={0001-8708},
     journal={Advances in Mathematics},
      volume={222},
      number={1},
       pages={240\ndash 254},
}

\bib{classification}{article}{
      author={Dobrinskaya, N.~E.},
       title={Classification problem for quasitoric manifolds over a given simple polytope},
        date={2001},
     journal={Funct. Anal. Appl.},
      volume={35},
      number={2},
       pages={83\ndash 89},
}

\bib{fakhruddin02}{misc}{
      author={Fakhruddin, N.},
       title={Multiplication maps of linear systems on smooth projective toric surfaces},
        date={2002},
        note={arXiv:math/0208178},
}

\bib{F26}{misc}{
      author={Ferroni, L.},
       title={Unimodality shenanigans in {E}hrhart theory},
        date={2026},
        note={arXiv:2609.10513},
}

\bib{fulton}{book}{
      author={Fulton, W.},
       title={Introduction to toric varieties},
   publisher={Princeton Univ. Press},
        date={1993},
}

\bib{firla1999hilbert}{article}{
      author={Firla, R.~T.},
      author={Ziegler, G.~M.},
       title={Hilbert bases, unimodular triangulations, and binary covers of rational polyhedral cones},
        date={1999},
     journal={Discrete Comput. Geom.},
      volume={2},
      number={21},
       pages={205\ndash 216},
}

\bib{Teddy}{misc}{
      author={Gonzales, T.},
      author={Lowen, C.},
       title={Structural properties of {B}ia{\l}ynicki-{B}irula decompositions},
        date={2026},
        note={arXiv:2604.27634},
}

\bib{grunbaum}{book}{
      author={Gr\"unbaum, B.},
       title={Convex polytopes},
     edition={Second},
      series={Grad. Texts in Math.},
   publisher={Springer},
        date={2003},
      volume={221},
        note={Prepared and with a preface by Volker Kaibel, Victor Klee and G\"unter M.\ Ziegler},
}

\bib{gubeladze2012convex}{article}{
      author={Gubeladze, Joseph},
       title={Convex normality of rational polytopes with long edges},
        date={2012},
     journal={Advances in Mathematics},
      volume={230},
      number={1},
       pages={372\ndash 389},
}

\bib{haase2010generating}{incollection}{
      author={Haase, C.},
      author={Lorenz, B.},
      author={Paffenholz, A.},
       title={Generating smooth lattice polytopes},
        date={2010},
   booktitle={Mathematical software---{ICMS} 2010},
      series={Lecture Notes in Comput. Sci.},
      volume={6327},
   publisher={Springer},
       pages={315\ndash 328},
}

\bib{Haase08}{article}{
      author={Haase, C.},
      author={Nill, B.},
      author={Paffenholz, A.},
      author={Santos, F.},
       title={Lattice points in {M}inkowski sums},
        date={2008},
     journal={The Electronic Journal of Combinatorics},
      volume={15},
      number={1},
}

\bib{haase21}{book}{
      author={Haase, C.},
      author={Paffenholz, A.},
      author={Piechnik, L.},
      author={Santos, F.},
       title={Existence of unimodular triangulations---positive results},
   publisher={American Mathematical Society},
        date={2021},
      volume={270},
      number={1321},
}

\bib{ikeda09}{article}{
      author={Ikeda, A.},
       title={Subvarieties of generic hypersurfaces in a nonsingular projective toric variety},
        date={2009},
     journal={Mathematische Zeitschrift},
      volume={263},
      number={4},
       pages={923\ndash 937},
}

\bib{Lundman_2012}{article}{
      author={Lundman, A.},
       title={A classification of smooth convex 3-polytopes with at most 16 lattice points},
        date={2012},
     journal={J. Algebr. Comb.},
      volume={37},
      number={1},
       pages={139\ndash 165},
}

\bib{Nakajima}{article}{
      author={Nakajima, H.},
       title={Affine torus embeddings which are complete intersections},
        date={1986},
     journal={Tohoku Math. J. (2)},
      volume={38},
      number={1},
       pages={85\ndash 98},
}

\bib{oda08}{article}{
      author={Oda, T.},
       title={Problems on {M}inkowski sums of convex lattice polytopes},
        date={2008},
        note={arXiv:0812.1418},
}

\bib{oda88}{book}{
      author={Oda, T.},
       title={Convex bodies and algebraic geometry},
      series={Ergebnisse der Mathematik und ihrer Grenzgebiete (3) [Results in Mathematics and Related Areas (3)]},
   publisher={Springer},
        date={1988},
      volume={15},
        note={Translated from the Japanese},
}

\bib{S86}{incollection}{
      author={Stanley, R.~P.},
       title={Log-concave and unimodal sequences in algebra, combinatorics, and geometry},
        date={1989},
   booktitle={Graph theory and its applications: {E}ast and {W}est ({J}inan, 1986)},
      series={Ann. New York Acad. Sci.},
      volume={576},
   publisher={New York Acad. Sci., New York},
       pages={500\ndash 535},
}

\bib{SV13}{article}{
      author={Schepers, J.},
      author={Van~Langenhoven, L.},
       title={Unimodality questions for integrally closed lattice polytopes},
        date={2013},
     journal={Ann. Comb.},
      volume={17},
      number={3},
       pages={571\ndash 589},
}

\bib{T23}{article}{
      author={Tsuchiya, A.},
       title={Cayley sums and {M}inkowski sums of lattice polytopes},
        date={2023},
     journal={SIAM J. Discrete Math.},
      volume={37},
      number={2},
       pages={1348\ndash 1357},
}

\bib{wiemeler15}{article}{
      author={Wiemeler, M.},
       title={Torus manifolds and non-negative curvature},
        date={2015},
     journal={Journal of the London Mathematical Society},
      volume={91},
      number={3},
       pages={667\ndash 692},
}

\bib{Yu21}{article}{
      author={Yu, Li},
      author={Masuda, Mikiya},
       title={On descriptions of products of simplices},
        date={2021},
     journal={Chinese Annals of Mathematics, Series B},
      volume={42},
      number={5},
       pages={777\ndash 790},
}

\bib{ziegler95}{book}{
      author={Ziegler, G.~M.},
       title={Lectures on polytopes},
      series={Grad. Texts in Math.},
   publisher={Springer},
        date={1995},
      volume={152},
}

\end{biblist}
\end{bibdiv}
\end{document}